\documentclass{amsart}

\usepackage[utf8]{inputenc}

\usepackage{tikz}
\usepackage{calc}
\usepackage[english]{babel}
\usepackage{cite}
\usepackage{color}
\usepackage{hyperref}
\usepackage{float}
\usepackage{tikz}
\usetikzlibrary{automata}
\usepackage{comment}
\usepackage{mathtools}
\usepackage{csquotes}

\usepackage{amsthm}
 \newtheorem{theorem}{Theorem}[section]
\theoremstyle{definition}
\newtheorem{definition}[theorem]{Definition}
\newtheorem{lemma}[theorem]{Lemma}

\newtheorem{notation}[theorem]{Notation}
\newtheorem{example}[theorem]{Example}
\newtheorem{conjecture}[theorem]{Conjecture}

\newtheorem{corollary}[theorem]{Corollary}

 \newtheorem{remark}[theorem]{Remark}
 
  \newtheorem{question}{Question}
\newtheorem{proposition}[theorem]{Proposition}

\usepackage{amssymb}

\tikzset{vertex/.style={circle, draw, fill=black!50},inner sep=0pt, minimum width=4pt}

\theoremstyle{definition}

\thanks {The author is supported by NSF grant DMS-1812021.}

\begin{document}

\title{A symbolic coding of the Morse Boundary}
\author{Abdalrazzaq Zalloum}
\begin{abstract}
Let $X$ be a proper geodesic metric space. We give a new construction of the Morse Boundary that realizes its points as equivalence classes of functions on $X$ which behave similar to the ``distance to a point" function. When $G=\langle S \rangle $ is a finitely generated group and $X=Cay(G,S)$, we use this construction to give a symbolic presentation of the Morse boundary as a space of ``derivatives" on $Cay(G,S).$ The collection of such derivatives naturally embeds in the shift space $\mathcal{A}^G$ for some finite set $\mathcal{A}.$

\end{abstract}

\maketitle

\section{Introduction} 
The Gromov boundary has been extremely fruitful in the study of hyperbolic groups. For example, if a group $G$ acts geometrically (properly and cocompactly) on two hyperbolic spaces $X$ and $Y$, then $\partial X \cong \partial Y$, making it possible to define the boundary of a group $G$ admitting such an action. However, as soon as non-negative (even zero) curvature is introduced, such theorems become harder to get. For instance, Croke and Kleiner \cite{CrokeKleiner} gave an example of a group $G$ acting geometrically on two homeomorphic CAT(0) spaces $X$ and $Y$ and yet yielding non-homeomorphic boundaries. A natural attempt to circumvent this problem would be to ignore the non-negatively curved parts of the space, and focus only on those negatively curved directions. This approach in fact works; Charney and Sultan \cite{ChSu2014} introduced the contracting boundary of a CAT(0) space, which is defined to be the collection of all geodesic rays in the CAT(0) space satisfying some hyperbolic-like condition, and they were able to show that it is a quasi-isometry invariant. This construction was generalized later by Cordes \cite{Cordes2017} to the Morse boundary of any proper geodesic metric space and therefore any finitely generated group. Since the Morse boundary of a geodesic metric space is a quasi-isometry invariant, it makes it a promising tool to distinguish groups up to quasi-isometry. This paper is devoted to giving a new construction of the Morse boundary.

One of the fundamental observations in geometric group theory, due to Schwarz  and Milnor, is that whenever two groups $G_1$ and $G_2$ act geometrically on a geodesic metric space $X$, then $G_1$ and $G_2$ must be quasi isometric to each others. For a large class of groups $\mathcal{H}$, knowing that a group $G$ is quasi isometric to a group $H$ with $H \in \mathcal{H}$, implies that $G$ and $H$ are in fact virtually isomorphic.
Examples of such groups are free abelian groups, surface groups, finite volume discrete subgroups of a non-compact Lie groups and many others. This is in fact quite powerful; it basically tells us that almost all of the properties of such an infinite group, are encoded in the geometry of the space that's being acted on by the group.

 Let $X$ be a proper geodeisc metric space. A geodesic ray $\alpha: [0, \infty) \rightarrow X$ is said to be Morse if it satisfies the hyperbolic-like property that quasi geodesics with end points on $Im(\alpha)$ stay close to $Im(\alpha)$. As a set, the Morse boundary, denoted by $\partial_{\star}X$, is defined to be the collection of equivalence classes of such geodesic rays where $\alpha \sim \beta$ if there exists some $C \geq 0$ such that $d(\alpha(t), \gamma(t)) \leq C$ for all $t \in [0, \infty)$. Based on ideas of Gromov, we describe a different construction of the Morse boundary as a certain collection of equivalence classes of $1$-Lipschitz maps from $X$ to $\mathbb{R}$ satisfying a certain distance-like condition. More precisely, if $\overline{H_{\star}}$ is the collection of all equivalence classes of distance-like functions having some Morse gradient ray, then we get a well defined $G$-equivariant continuous surjection from $\overline{H_{\star}}$ onto $\partial_{\star}X$, the Morse boundary of $X$, recovering the Morse boundary as $\partial_{\star} X \cong \overline{H_{\star}}/ \sim$

\begin{theorem}

Let $X$ be a proper geodesic metric space. Then there exists a subset $H_{\star}$ of Lip$(X, \mathbb{R})$ along with a continuous surjection $\varphi:\overline{H_{\star}} \twoheadrightarrow \partial_{\star}X$. In particular, we have $\partial_{\star} X \cong \overline{H_{\star}}/ \sim$
\end{theorem}

Symbolic dynamics studies the action of a group $G$ on $\mathcal{A}^G,$ the set of maps from this group to a finite set $\mathcal{A}$, which is called the set of symbols.  Such an action is called a Bernoulli shift. In his famous paper \cite{Gromov1987}, Gromov shows that, when $G$ is a hyperbolic group, there is a finite set of symbols $\mathcal{A}$ , a subset $Y \subseteq \mathcal{A}^G$ along with a Bernoulli shift action of $G$ on $Y$ such that if $\partial G$ is the Gromov boundary of $G$, then there exists a $G$-equivariant continuous surjection $\varphi:Y \twoheadrightarrow \partial G$.  This is called a \emph{symbolic coding} or \emph{symbolic presentation} of $\partial G$. In other words, if a group $G$ acts on a topological space $X$, we say that this action admits a symbolic coding if there exists a finite set $\mathcal{A}$, some $G$-equivariant subset $Y$ of $\mathcal{A}^G$ along with a continuous surjection of $Y$ onto $X$ which is $G$-equivariant. 

It is known that an action of a discrete group on a Polish topologcial space, which is a certain kind of metrizable space, admits a symolic coding (Theorem 1.4 of \cite{Seward2014}). However, the Morse boundary is not metrizable or even second countable \cite{Murray2015}. We use the new construction of the Morse boundary in the previous theorem to show that

\begin{theorem}
Let $G$ be a finitely generated group. Then, the action of $G$ on its Morse boundary admits a symbolic coding. In other words, there exists a finite set of symbols $\mathcal{A},$ a subset $Y_{\star}$ of $\mathcal{A}^G$ along with a Bernoulli shift action of $G$ on $Y_{\star}$; such that if $\partial_{\star} G$ is the Morse boundary of $G$, then there exists a $G$-equivariant continuous surjection 
$\phi: Y_{\star} \rightarrow \partial_{\star} G$
\end{theorem}

This construction opens the door to using symbolic dynamical methods to understand the Morse boundary. For example, when $G$ is a hyperbolic group, this construction has been studied further by Coornaert and Papadopoulos \cite{Coornaert2001} where they show that $Y \subseteq \mathcal{A}^G$ is a subshift of finite type. Also Cohen, Goodman-Strauss and Rieck \cite{cohen2018} used a similar coding to the one in the above theorem to show that a hyperbolic group admits a strongly aperiodic subshift of finite type if
and only if it has at most one end. The paper is organized as follows

In section 2, we give a sketch of the proof.

In section 3, we define distance-like functions and we show how to assign a gradient ray to each such function

In section 4, we state and prove the first main theorem. We also give a non-example for the case where the gradient rays are not Morse. In other words, we give an example of two $h$-gradient rays $\alpha$ and $\beta$ where $d(\alpha(t), \beta(t)$ is unbounded.

In section 5, we introduced few objects from Symbolic Dynamics and we show how the first main theorem can be refined to imply the second.

In section 6, we study the behaviour of Busemann functions whose defining rays are Morse and we show that level sets of these functions, have very similar behaviour to horocycles in the hyperbolic space $\mathbb{H}^n$.

The author would like to thank his main advisor Johanna Mangahas and his coadivisor Ruth Charney for their exceptional support and guidance. He would also like to thank Bernard Badzioch, David Cohen, Joshua Eike and Devin Murray for helpful conversations.

\section{ Sketch of the proof}
\begin{figure} 
    \centering
    \label{fig: Morse}

\begin{tikzpicture}[scale=.8]

\draw[very thick,red] (0,0) circle (3cm);

\draw[latex-,very thick,red] (0,3) .. controls ++(0,-1) and ++({-2*cos(-45)},{-2*sin(-45)}) ..({2*cos(-45)},{2*sin(-45)});
\draw[very thick,red] (0,3) .. controls ++(0,-1) and ++({-2*cos(-135)},{-2*sin(-135)}) ..({2*cos(-135)},{2*sin(-135)});

\draw[very thick,red] (0,3) .. controls ++(0,-1) and ++({-1.1*cos(-160)},{-1.1*sin(-160)}) ..({2*cos(-160)},{2*sin(-160)});

\draw[very thick,red] (0,3) .. controls ++(0,-1) and ++({-1.1*cos(-20)},{-1.1*sin(-20)}) ..({2*cos(-20)},{2*sin(-20)});

%\draw[very thick,blue] (0,3) .. controls ++(0,-2) and ++(-2,0) ..(3,0);%

\draw[very thick,red] (0,3) .. controls ++(0,-2) and ++({-2*cos(-120)},{-2*sin(-120)}) ..({2*cos(-120)},{2*sin(-120)});
\draw[very thick,red] (0,3) .. controls ++(0,-2) and ++({-2*cos(-70)},{-2*sin(-70)}) ..({2*cos(-70)},{2*sin(-70)});

\draw[very thick,red] (0,3) .. controls ++(0,-1) and ++({-1.1*cos(-180)},{-1.1*sin(-180)}) ..({2*cos(-180)},{2*sin(-180)});
\draw[very thick,red, latex-] (0,3) .. controls ++(0,-1) and ++({-1.1*cos(0)},{-1.1*sin(0)}) ..({2*cos(0)},{2*sin(0)});

\draw[very thick,blue, latex-] ({3*cos(120)},{3*sin(120)}) .. controls ++({-2*cos(120)},{-2*sin(120)}) and ++(0,1) ..(-.5,-2);

\draw[very thick,blue, latex-] ({3*cos(60)},{3*sin(60)}) .. controls ++({-2*cos(60)},{-2*sin(60)}) and ++(0,1) ..(.5,-2);

\draw[thick,fill=black] (0,3) circle (.1cm);

\node[above] at (.5,3) {$\zeta$};

\end{tikzpicture}

\caption{The set of blue and red gradient rays are both $h$-gradient rays. The red ones are Morse and therefore they all converge to the same point $\zeta$ in the Morse boundary}

\end{figure}
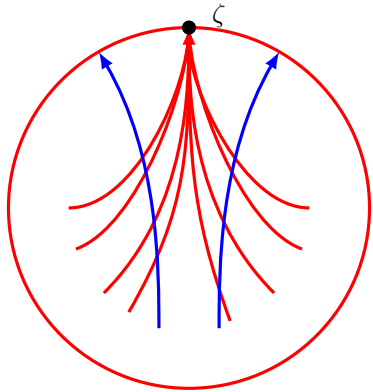

We now give an informal sketch of the proof. But first, I will describe the proof for the case where the space $X$ is $\delta$-hyperbolic (this can be found in \cite{Gromov1987} and has also been explained further in \cite{michelcoornaert1993} and \cite{Coornaert2001}). When $X$ is $\delta$-hyperbolic, we start with a collection of $1$-Lipschitz maps $h:X \rightarrow \mathbb{R}$ which satisfy a distance-like property in the sense that the distance between a point $x$ in $X$ to a level set $h^{-1}( \lambda)$ is the same as the difference $|h(x)-\lambda|$, we denote the collection of such maps by $H$. To each such map $h \in H$, and for any $x \in X$, there is a natural way to define a geodesic ray starting at $x$, called a gradient ray corresponding to $h$, or an $h$-gradient ray. This can be done by taking successive projections on descending level sets. In other words, if $x$ is in the level set of $\lambda_1$ and $\lambda_2 < \lambda_1$, then we can take a projection of $x$ onto the level sets of $\lambda_2$. Now, we can iterate this process to obtain a gradient ray $g$ starting at $x$. We remark that this gradient ray is ``perpendicular" to all of the level sets it crosses in the sense that for any $s, t \in [0,\infty)$, we have $d(h(g(s), h(g(t))=t-s,$ which justifies the use of the word ``gradient" for the geodesic ray $g$. The above suggests the possibility of a map from $H$ to $\partial X$, the Gromov boundary of $X$. Notice that we don't have such a map yet, since we might have two $h$-gradient rays, $\alpha$ and $\beta$ for the same $h \in H,$ such that $d(\alpha(t), \beta(t))$ is unbounded. However, the goemetry of $\delta$-hyperbolic spaces ensures this can't happen (see \cite{Coornaert2001} for example), and therefore, we do get a well defined map from $H$ to $\partial X.$ It is then shown that this map is surjective, this is done in \cite{Gromov1987}, \cite{michelcoornaert1993} and \cite{Coornaert2001}.
\begin{figure}

 \centering

    \label{fig: derivative}
    \vspace{5mm}
    
    \begin{tikzpicture}[scale=.7]

\draw[thick,fill=black] (-2,2) circle (0.1cm); 
\node[right] at (-1.8,2) {$1$};
\draw[thick,fill=black] (-2,4) circle (0.1cm);
\node[right] at (-1.8,4) {$1$};
\draw[thick,fill=black] (-2,6) circle (0.1cm);
\node[right] at (-1.8,6) {$1$};
\draw[thick,fill=black] (-2,8) circle (0.1); 
\node[right] at (-1.8,8) {$1$};
\draw[thick,fill=black] (-4,2) circle (0.1);
\node[right] at (-3.8,2) {$2$};
\draw[thick,fill=black] (-4,4) circle (0.1cm);
\node[right] at (-3.8,4) {$2$};

\draw[thick,fill=black] (-4,6) circle (0.1cm);
\node[right] at (-3.8,6) {$2$};
\draw[thick,fill=black] (-4,8) circle (0.1cm);
\node[right] at (-3.8,8) {$2$};

\draw[thick,fill=black] (-6,2) circle (0.1cm);
\node[right] at (-5.8,2) {$3$};
\draw[thick,fill=black] (-6,4) circle (0.1cm);
\node[right] at (-5.8,4) {$3$};
\draw[thick,fill=black] (-6,6) circle (0.1cm);
\node[right] at (-5.8,6) {$3$};
\draw[thick,fill=black] (-6,8) circle (0.1cm);
\node[right] at (-5.8,8) {$3$};

\draw[thick,fill=black] (-8,2) circle (0.1cm);
\node[right] at (-7.8,2) {$4$};

\draw[thick,fill=black] (-8,4) circle (0.1cm);
\node[right] at (-7.8,4) {$4$};

\draw[thick,fill=black] (-8,6) circle (0.1cm);
\node[right] at (-7.8,6) {$4$};

\draw[thick,fill=black] (-8,8) circle (0.1cm);
\node[right] at (-7.8,8) {$4$};

\draw[<->] (2,1.75) node[below] {$ $} -- (2,2.25)node[above=1]{$0$};
\draw[<->] (1.75,2) -- (2.25,2) node[right=1] { -1};

\draw[<->] (2,3.75) node[below=1] {$0$} -- (2,4.25)node[above=1]{$0$};
\draw[<->] (1.75,4) -- (2.25,4) node[right=1] { -1};

\draw[<->] (2,5.75) node[below=1] {$0$} -- (2,6.25)node[above=4]{$0$};
\draw[<->] (1.75,6) -- (2.25,6) node[right=1] { -1};

\draw[<->] (2,7.75) node[below=1] {$0$} -- (2,8.25)node[above=1]{$ $};
\draw[<->] (1.75,8) -- (2.25,8) node[right=1] { -1};

\draw[<->] (4,1.75) node[below=1] {$ $} -- (4,2.25)node[above=1]{$0$};
\draw[<->] (3.75,2) node[left=4] { 1} -- (4.25,2) node[right=4] { -1};

\draw[<->] (4,3.75) node[below=1] {$0$} -- (4,4.25)node[above=1]{$0$};
\draw[<->] (3.75,4) node[left=1] { 1} -- (4.25,4) node[right=1] { -1};

\draw[<->] (4,5.75) node[below=1] {$0$} -- (4,6.25)node[above=1]{$0$};
\draw[<->] (3.75,6) node[left=4] { 1} -- (4.25,6) node[right=1] { -1};

\draw[<->] (4,7.75) node[below=1] {$0$} -- (4,8.25)node[above=1]{$ $};
\draw[<->] (3.75,8) node[left=1] { 1} -- (4.25,8) node[right=1] { -1};

\draw[<->] (6,1.75) node[below=1] {$ $} -- (6,2.25)node[above=1]{$0$};
\draw[<->] (5.75,2) node[left=1] { 1} -- (6.25,2) node[right=1] { -1};

\draw[<->] (6,3.75) node[below=1] {$0 $} -- (6,4.25)node[above=1]{$0$};
\draw[<->] (5.75,4) node[left=1] { 1} -- (6.25,4) node[right=4] { -1};

\draw[<->] (6,5.75) node[below=1] {$0 $} -- (6,6.25)node[above=1]{$0$};
\draw[<->] (5.75,6) node[left=1] { 1} -- (6.25,6) node[right=1] { -1};

\draw[<->] (6,7.75) node[below=1] {$0 $} -- (6,8.25)node[above=1]{$ $};
\draw[<->] (5.75,8) node[left=1] { 1} -- (6.25,8) node[right=1] { -1};

\draw[<->] (8,1.75) node[below=1] {$ $} -- (8,2.25)node[above=1]{$0$};
\draw[<->] (7.75,2) node[left=1] { 1} -- (8.25,2) node[right=1] {  };

\draw[<->] (8,3.75) node[below=1] {$0 $} -- (8,4.25)node[above=1]{$0$};
\draw[<->] (7.75,4) node[left=1] { 1} -- (8.25,4) node[right=4] {  };

\draw[<->] (8,5.75) node[below=1] {$0 $} -- (8,6.25)node[above=1]{$0$};
\draw[<->] (7.75,6) node[left=1] { 1} -- (8.25,6) node[right=4] {  };

\draw[<->] (8,7.75) node[below=1] {$ 0$} -- (8,8.25)node[above=1]{$ $};
\draw[<->] (7.75,8) node[left=1] { 1} -- (8.25,8) node[right=1] {  };

\end{tikzpicture}
 \caption{To the left, a 1-Lipschitz map $h:\mathbb{Z \oplus Z} \rightarrow \mathbb{Z}$, to the right, its derivative $dh$}

\end{figure}

If $X$ is a proper geodesic metric space, we work out a similar construction for the Morse boundary of the space. In other words, we start with a collection of $1$-Lipschitz maps $h:X \rightarrow \mathbb{R}$ which satisfy a distance-like property. To each such map $h$, and for any $x \in X$, we can still define a gradient ray starting at $x$ by taking successive projections on descending level sets. We then define $H_{\star}$ to be the collection of all such maps having some Morse gradient ray at some point.  A key result of this paper is to show that if $\alpha$ and $\beta$ are two Morse rays (with possibly different starting points and different Morse gauges) associated to such a map $h$, then they must define the same point in the Morse boundary. In other words, we show that if $h \in H_{\star}$ and $\alpha$ and $\beta$ are two Morse gradient rays associated to $h$, then there exists a $C \geq 0$ such that $d(\alpha(t), \beta(t)) \leq C$ for at $t \in [0, \infty).$ This can be interpreted as follows: To each distance-like map $h \in H_{\star}$, one can assign an infinite collection of gradient rays, some of which are Morse and some are not. Whereas those gradient rays do not all converge to the same point in general, the Morse ones do, and hence defining a unique point in the Morse boundary $\partial_{\star}X$ (see Figure 1). So although in general we can't get a well defined map from the space of such functions to whatever notion of a boundary a space might have; if we restrict our attention to the Morse boundary such a map exists.

The above yields a well defined map $\varphi:H_{\star} \rightarrow \partial_{\star} X$. This map is then shown to be surjective using Busemann functions. More precisely, if $c$ is a Morse ray, then the corresponding Buseman function $b_c(x)=\lim [d(x,c(t))-t]$ is in $H_{\star}$ having $c$ as its Morse gradient ray. This proves the first main theorem.

For the second theorem, we do the the following. Let $G=\langle S \rangle $ be a finitely generated group and let $X=Cay(G,S).$ We first consider the subspace $H_0 \subseteq H_{\star}$ whose elements $h$ take integer values on the vertices of $X$. In other words, $H_{0}:=\{h \in H_{\star} | \,h(X^{0}) \subseteq \mathbb{Z} \}.$ We then show that this set is still big enough to surject on the Morse boundary by showing that for any Morse geodesic $c:[0, \infty) \rightarrow X$, starting at a vertex in $X$, the corresponding Busemann function $b_c \in H_{0}.$ Since maps $h \in H_0$ take integer values on vertices, the subspace $H_0$ can be thought of as a subspace of Lip$(G, \mathbb{Z})$ rather than Lip$(X, \mathbb{R}).$ But then we show that each such map $h \in H_0 \subseteq \text{Lip}(G, \mathbb{Z})$ is completely determined by a map $dh:G \rightarrow \mathcal{A}$ for some finite set $\mathcal{A}.$ Therefore, the space $H_0$, which surjects on the Morse boundary, has a natural identification with a subset $Y \subseteq \mathcal{A}^G$ giving the second theorem. The only part left to explain is, given $h \in H_0 \subseteq \text{Lip}(G, \mathbb{Z})$, how can we get a map from $G$ to a finite set $\mathcal{A}$ that completely determines $h$? This is done as follows. First remember that $G$ is generated by $S$. For any  $h \in H_0 \subseteq \text{Lip}(G, \mathbb{Z})$, we define $dh:G \rightarrow \mathcal{A}$ by $$g \mapsto (s \mapsto h(gs)-h(g))$$ where $\mathcal{A}=\{-1, 0, 1\}^S$

Consider Figure 2. To the left, we have a 1-Lipschitz map defined on $h:\mathbb{Z \oplus Z} \rightarrow \mathbb{Z}$, and on the right we have it's derivative $dh:\mathbb{Z \oplus Z} \rightarrow \mathcal{A}$. Notice that $h$ is compeltely determined by $dh$ and by the value of $h$ at one point $p$.

\vspace{2mm}

More generally, if $G_1$ and $G_2$ are groups generated by the symmetrized finite sets $S_1$ and $S_2$ and $h \in \text{Lip}(G_1, G_2)$ is a 1-Lipschitz map, then if we let $\mathcal{A}=(S_2 \cup \{1_{G_2}\})^{S_1}$ we get a map $dh:G_1 \rightarrow \mathcal{A}$ given by $$g \mapsto (s \rightarrow h^{-1}(g)h(gs))$$

This notion of a derivative was introduced by  Cohen \cite{Cohen2017}. In fact, he proved that $\{dh| \, h \in \text{Lip}(G_1, G_2)\} \subseteq \mathcal{A}^{G_1}$ is a subshift of finite type.

\section{Constructing gradient rays} The main goal of this section is to define distance-like functions and to show that there is a natural way of associating a gradient ray to each such function.

Unless mentioned otherwise, for the rest of the paper, $X$ will be a proper geodesic metric space.

\begin{definition}
Let $c:[0,\infty] \rightarrow X$. Define the Busemann function associated to $c$ by $$b_c(x)=\lim_{t\rightarrow \infty} [d(x,c(t))-t]$$ 
\end{definition}

We remark that the above limit exist by the triangular inequality. Also, notice that a Busemann function is 1-Lipschitz.

\begin{definition}[distance-like function]\label{def: distance-like}
A continuous map $h:X \rightarrow \mathbb{R}$ is said to be distance-like if whenever $h(x) \geq \lambda$, for some $\lambda \in \mathbb{R}$, then $h^{-1}(\lambda)$ is nonempty and $h(x)=\lambda +d(x,h^{-1}(\lambda))$
\end{definition}
We remark that it is implied by the definition that if $a \in im(h)$ then so is $b$ for any $b \leq a.$ Denote the space of all 1-Lipschitz functions on $X$ by Lip($X, \mathbb{R})$.

\begin{proposition}\label{Lip}

Any distance-like function $h$ is in Lip($X, \mathbb{R})$.
\end{proposition}

\begin{proof}
The proof follows easily from the definition.
\end{proof}

The following proposition states that for any $x \in X$, if $\lambda< h(x),$ then there exist a a point $y$ in the level set $h^{-1}(\lambda),$ realizing the distance between $x$ and the level set $h^{-1}(\lambda).$ Therefore, one gets a geodesic connecting $x$ to $h^{-1}(\lambda)$ realizing the distance $d(x, h^{-1}(\lambda)).$ We will eventually iterate this process to obtain a geodesic ray starting at $x$ which is ``perpendicular" to all the $h$-level sets it crosses

\begin{proposition}\label{funny}
If $h$ is a distance-like function and $x \in X$, $\lambda \in \mathbb{R}$ with $h(x)> \lambda$, then $\exists y \in X$ with $d(x,y)=h(x)-h(y)$, where $h(y)=\lambda$
\end{proposition}

\begin{proof}
Notice that since $h$ is continuous, $h^{-1}(\lambda)$ must be a closed set. Set $Y=h^{-1}(\lambda)$. Now, by definition, we have $d(x,Y)=$inf$\{d(x,y)|\,y \in Y\}.$ This yields a sequence of points $y_n \in Y$ such that $d(x,y_n) \rightarrow d(x,Y).$ Choose $r \in \mathbb{R}$ such that $r>d(x,y_n)$ for all $n$ (this is possible since $d(x,y_n)$ is a convergent sequence). Notice that the closed ball $B(x,r)$ contains $y_n$ for all $n$, and since $X$ is proper $B(x,r)$ must be compact. Therefore, the sequence $y_n$ has a convergent subsequence $y_{n_{k}} \rightarrow y$ with $y \in Y$ as $Y$ is closed. Now, since $d(x,-)$ is continuous, we must have $d(x,y_{n_{k}}) \rightarrow d(x,y)$, but since we also have $d(x,y_{n_{k}}) \rightarrow d(x,Y)$, we get $d(x,Y)=d(x,y)$ with $y \in Y.$ This implies the conclusion of the proposition.
\end{proof}

\begin{proposition}
If $h$ is a Busemann function and $h(x) \geq \lambda$ for some $x \in X$ and $\lambda \in \mathbb{R}$, then there must exist $ p \in X$ with $h(x)= \lambda+d(x,p).$ Furthermore, $h(p)=\lambda$
\end{proposition}
\begin{proof}
This Proposition is Proposition 3.4 in \cite{michelcoornaert1993}, they assume that $X$ is $\delta$-hyperbolic and they use a different definition for a distance-like function but their argument is still valid in our settings.

\end{proof}
Now we use the above Proposition to show that any Busemann function is distance-like

\begin{proposition}
Any Busemann function is distance-like.
\end{proposition}
\begin{proof}
Let $h$ be a Busemann function and let $x \in X$, $\lambda \in \mathbb{R}$ with $h(x) \geq \lambda $. By the previous Proposition, there must exist some $p \in X$ such that $h(x)=\lambda+d(x,p)$ with $h(p)= \lambda$. Therefore, we need only to show that $d(x,p)=d(x,h^{-1}( \lambda))$. But since a Busemann function is 1-Lipschitz, any $p' \in X$ with $h(p')=\lambda$ must satisfy $h(x)-h(p')=d(x,p) \leq d(x,p')$. Therefore, $d(x,p)=d(x,h^{-1}(\lambda)).$
\end{proof}

\begin{definition}[gradient arc]\label{def: gradient arc}
Let $h$ be a distance-like function, an $h$-gradient arc is a path $g: I \rightarrow X$ parametrized by the arc's length such that $h(g(t))-h(g(s))=s-t$ for all $s, t \in I$. An $h$-gradient ray is a geodesic ray whose restriction to any interval $I \subset [0 ,\infty)$ is an $h$-gradient arc
\end{definition}

The following three lemma's are stated and proved in\cite{Coornaert2001} for the  case where $X$ is $\delta$-hyperbolic. However, they only use that $X$ is a proper geodesic metric space which is the case here

\begin{lemma}[concatenation of gradient arcs]
Let $I_1$ and $I_2$ be closed intervals of the real line such that $I_2$ begins where $I_1$ ends. If $I=I_1 \cup I_2$ and $g:I \rightarrow X$ is a path whose restrictions to $I_1$ and $I_2$ are $h$-gradient arcs, then $g$ itself is an $h$-gradient arc.
\end{lemma}

\begin{lemma}(characterization of $h$-gradient arcs)
Let $X$ be a proper geodesic metric space and let $h$ be a distance-like function:
\begin{itemize}
    \item If $g$ is an $h$-gradient ray then $g$ is a geodesic
    \item If $x,  y \in X$ are points so that $h(x)-h(y)=d(x,y)$, and $g$ is a geodesic connecting $x$ to $y$ then $g$ is an $h$-gradient arc.
\end{itemize}

\end{lemma}

\begin{figure}
    \centering
    
    \label{fig: horospheres}
\begin{tikzpicture}

\draw[very thick,red] (4,0) .. controls ++({-4*cos(30)},{4*sin(30)}) and ++({4*cos(30)},{4*sin(30)}) ..(-4,0);

\draw[very thick,red] (4,3) .. controls ++({-4*cos(30)},{4*sin(30)}) and ++({4*cos(30)},{4*sin(30)}) ..(-4,3);

\draw[very thick,red] (4,2) .. controls ++({-4*cos(30)},{4*sin(30)}) and ++({4*cos(30)},{4*sin(30)}) ..(-4,2);
\draw[very thick,red] (4,1) .. controls ++({-4*cos(30)},{4*sin(30)}) and ++({4*cos(30)},{4*sin(30)}) ..(-4,1);

\draw[thick,fill=black] (0,1.5) circle (.06cm); \node[right] at (0,1.2) {$p$};
\draw[thick,fill=black] (0,2.5) circle (.06cm); \node[right] at (0,2.2) {$p_1$};
\draw[thick,fill=black] (0,3.5) circle (.06cm); \node[right] at (0,3.2) {$p_2$};
\draw[thick,fill=black] (0,4.5) circle (.06cm); \node[right] at (0,4.2) {$p_3$};

\draw[very thick,blue, -latex] (0,1.5) -- (0,5.3);
\node[right] at (.1,5.1) {$g(t)$};

\end{tikzpicture}

\caption{Constructing a gradient ray}

\end{figure}
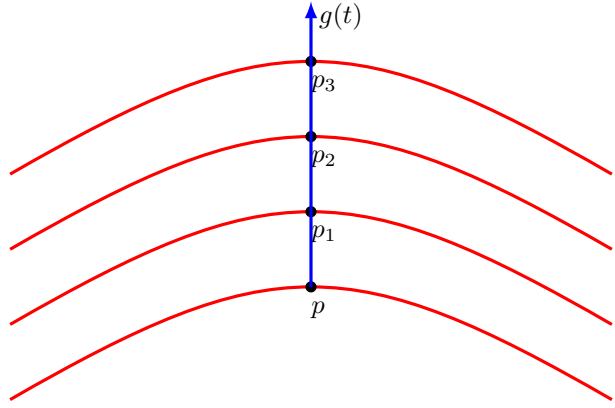

\begin{lemma}[gradient rays]
For any distance-like function $h$ and any $p \in X$, there exists an $h$-gradient ray starting at $p$
\end{lemma}

\begin{proof}
This follows from Lemmas 3.4, 3.8 and 3.9 of this section. The idea of the proof is to set $\lambda_1=h(p)-1$, take a projection of $p$ on $h^{-1}( \lambda_1)$ and call this projection $p_1.$ And then we repeat(see Figure 3), in other words, we set $\lambda_2=h(p_1)-1$, take a projection of $p_1$ on $h^{-1}( \lambda_2)$ and call this projection $p_2.$ This process, yields an $h$-gradient ray $g$. Fore more details, see Proposition 2.13 in \cite{Coornaert2001}. 
\end{proof}

\section{Main Result}

In section, we prove the first main result of this paper. The following definition is a coarse version of a convex functions

\begin{definition}

A map $h:X \rightarrow \mathbb{R}$ is said to be $K$-convex if there exists some $K>0$ such that given any geodesic $[x_0,x_1]$ and  $t \in [0,1]$, if $x_t $  satisfies $d(x_0,x_t)=td(x_0,x_1)$ then we must have $$h(x_t) \leq (1-t)h(x_0)+th(x_1)+K$$
\end{definition}

\begin{definition}Let $X$ be a geodesic metric space and let $\delta \geq 0$. Given $x,y,z \in X$. Let $\Delta:=[x,y] \cup [y,z] \cup [x,z]$ be a geodesic triangle with vertices $x,y$ and $z.$ We say that $\Delta$ is $\delta$-thin if the union of the $\delta$-neighborhoods of any two sides of the triangle contains the third. We refer to this condition as \emph{slim condition 1}
\end{definition}

\begin{remark}\label{tripod} For any proper geodesic metric space $X$ and any $x,y,z \in X$, if $\Delta:=[x,y] \cup [y,z] \cup [x,z]$ is a geodesic triangle with vertices $x,y$ and $z,$ there must exist three non-negative integers $a,b,c$ and three points $i_x, i_y, i_z$ such that $a+b=|[x,y]|$, $a+c=|[x,z]|$ and $b+c=|[y,z]|$ where $i_x \in [y,z]$, $i_y \in [x,z]$ and $i_z \in [x,y]$ satisfying that $a=d(x,i_z)=d(x,i_y)$, $b=d(y,i_z)=d(y,i_x)$ and $c=d(z,i_x)=d(z,i_y).$ The points $i_x, i_y$ and $i_z$ are called the \emph{internal} points of $\Delta$.
\end{remark}

\begin{definition}Let $X$ be a geodesic metric space and let $\delta' \geq 0$. Given $x,y,z \in X$, define $\Delta:=[x,y] \cup [y,z] \cup [x,z]$ be a geodesic triangle with vertices $x,y$ and $z.$ We say that $\Delta$ is $\delta'$-thin provided that for any point $w\in \{x,y,z\}$, if we consider the two subgeodesics $\alpha_1$ and $\alpha_2$ connecting $w$ to the internal points on those two geodesics, then whenever $u \in im(\alpha_1)$, there must exist some $v \in im(\alpha_2)$ with $d(u,v) \leq \delta'$. We refer to this condition as \emph{slim condition 2} 
\end{definition}

It is clear that if a triangle is $\delta'$-slim with slim condition 2, then it is $\delta$-slim with slim condition 1. But we also have:
\begin{lemma}\label{equivalence}
Slim condition 1 and slim condition 2 are equivalent.
\end{lemma}

\begin{proof}

The argument in Proposition 1.17 in the \emph{Reformulations of the Hyperbolicity} section of \cite{BH} shows the above Lemma provided that $X$ is a $\delta$-hyperboic metric space, however, since their argument doesn't use thinness of any triangles in $X$ aside from the $x,y,z$ one, the argument is still valid here.

\end{proof}

\begin{lemma}\label{convex}
Let $X$ be a geodesic metric space space, and let $x, y, z$ be the vertices of a $\delta'$-thin triangle in the sense of slim condition 2. Then for any $t \in [0,1]$ and any $x_t \in [y,z]$ with $d(y,x_t)=td(y,z),$ we must have 

$$d(x,x_t) \leq (1-t)d(x,y)+td(x,z)+2\delta'$$
\end{lemma}

\begin{proof} Let $i_x, i_y, i_z$ and $a,b,c$ be as in Remark \ref{tripod}. Let $x_t \in [y,z],$ with $d(y,x_t)=td(y,z).$ Then $x_t \in [y,i_x] \cup [i_x,z]$. Therefore, we have two cases to consider. If $x_t \in [y,i_x],$ then by $\delta'$-slimness, we must have some $w \in [y,i_z]$ such that $d(x_t,w) \leq \delta'$. This implies using triangular inequality imples that

$$d(y,x_t)- \delta' \leq d(y,w) \leq d(y, x_t)+\delta'$$

$$d(x,w)-\delta' \leq d(x,x_t) \leq d(x,w)+ \delta'$$

By the above two equations, we get that
\begin{align*}
    d(x,x_t)&\leq d(x,w)+ \delta' \\
    &=(a+b)-d(y,w)+\delta'\\
    &\leq (a+b)+\delta'-d(y,x_t)+\delta'\\
    &=(a+b)-t(b+c)+2\delta'\\
    &=a+(1-t)b-tc+2\delta' \\
    &\leq a+(1-t)b+tc+ 2\delta' \\
    &=(1-t)(a+b)+t(a+c)+ 2\delta'\\
\end{align*}

Now for the second case, if $x_t \in [i_x, z],$ by $\delta'$-slimness of the triangle, there must exist some $u \in [z,i_y]$ with $d(x_t,u) \leq \delta'$. By the triangular inequality, we get that 

$$d(z,x_t) -\delta' \leq d(z,u) \leq d(z, x_t) +\delta'$$

$$d(x, u)- \delta' \leq d(x,x_t) \leq d(x,u)+ \delta$$

Therefore, we get that 
\begin{align*}
    d(x,x_t)&\leq d(x,u)+ \delta' \\
    &=(a+c)-d(z,u)+\delta' \\
    &\leq (a+c) +\delta'-d(z,x_t)+\delta' \\
    &=(a+c)-(1-t)(b+c)+2\delta'\\
    &=a-(1-t)b+tc+ 2\delta' \\
    &\leq a+(1-t)b+tc+ 2\delta'\\
    &=(1-t)(a+b)+t(a+c)+2\delta'\\
\end{align*}

\end{proof}

Let $\text{Lip}(X, \mathbb{R})$ be defined to be the set of all 1-Lipschitz maps from $X$ to $\mathbb{R}$. Let $\overline{\text{Lip}(X, \mathbb{R})}$ denote the quotient of $\text{Lip}(X, \mathbb{R})$ by all constants Lipschitz maps. Notice that for a fixed $p$, the space $\overline{\text{Lip}(X, \mathbb{R})}$ is homeomorphic to the space $\text{Lip}_{p}(X, \mathbb{R})$ consisting of all $f \in \text{Lip}(X, \mathbb{R})$ such that $f(p)=0.$ It is an easy exercise to see that the space $\text{Lip}_{p}(X, \mathbb{R})$ is compact (alternatively, this is Proposition 3.1 in \cite{Maher2015RandomWO}). Notice that if $h_1,$ $h_2$ are distance-like function that differ by a constant, then $g$ is a gradient ray for $h_1$ if and only if it is a gradient ray for $h_2.$ Therefore, it is more natural to think of gradient rays as being associated to $\overline{h}$ rather than being associated to a specific $h$.

\begin{definition}[Morse] A geodesic $\gamma$ in a metric space
is called $N$-Morse, where $N$ is a function $[1, \infty) \times [0, \infty) \rightarrow [0, \infty)$, if for any
$(\lambda, \epsilon)$-quasi-geodesic $\sigma$ with endpoints on $\gamma$, we have $\sigma \subseteq \mathcal{N}_{N(\lambda,\epsilon)}(\gamma)$. The
function $N(\lambda, \epsilon)$ is called a Morse gauge. A geodesic $\gamma$ is said to be Morse if it is $N$-Morse for some $N.$
\end{definition}

The following is a key Lemma; we will use it to show that the space of distance-like functions having a Morse gradient ray is $G$-equivariant where $G$ is any finitely generated group.

\begin{lemma} \label{AllMorse}
Let $p, q \in X$ and let $h$ be a distance-like function. Then $h$ has a Morse gradient ray starting at $p$ if and only if it has a Morse gradient ray starting at $q$.
\end{lemma}

\begin{proof} 

Let $p, q \in X$ and let $g$ be a Morse $h$-gradient ray starting at $p$. Notice that since $h(g(t))-h(g(0))=0-t$, then we have that $h(g(t)) \rightarrow -\infty $ as $t \rightarrow \infty.$ In other words, for any $r \in \mathbb{R},$ there exists some $t$ such that $h(g(t)) < r.$ This allows us to construct a sequence of gradient arcs $\{\alpha_n\}_{n \in \mathbb{N}}$ joining $q$ and $g(n).$ By By Arzel`a-Ascoli, there exists a subsequence of $\{\alpha_{n_k}\}_{k \in \mathbb{N}}$ that converges on compact sets to a geodesic ray $\alpha$. This geodesic ray is Morse by Lemma 2.8 in \cite{Cordes2017}. We need to show that this geodesic ray is in fact an $h$-gradient ray. In other words, we should show that for any $s, t \in [0, \infty),$ we must have $h(g(s))-h(g(t))=t-s.$ Fix $s, t \in [0, \infty)$ and let $K$ be a compact set containing both $s$ and $t$. Since $\{\alpha_{n_k}\}_{k \in \mathbb{N}}$ converges to $\alpha$ uniformly on $K$, we get that $\alpha_{n_k}(s) \rightarrow \alpha(s)$ and $\alpha_{n_k}(t) \rightarrow \alpha(t)$. But since $h$ is continuous, we get that $h(\alpha_{n_k}(s)) \rightarrow h(\alpha(s))$ and $h(\alpha_{n_k}(t)) \rightarrow h(\alpha(t))$. Hence, since $\alpha_{n_k}$ are all $h$-gradient arcs, we get that for any $\epsilon>0,$ we have $|[h(\alpha(t)-h(\alpha(s)]-[s-t]|=|[h(\alpha(t)-h(\alpha(s)]-[h(\alpha_{n_k}(t)-h(\alpha_{n_k}(s)]| < \epsilon$, which establishes the desired conclusion.

\end{proof}

\begin{definition}[Morse Convex]\label{Morse-Convex} A distance-like function $h$ is said to be \emph{ $N$-Morse convex} if for any $N$-Morse rays $\alpha$,  $\beta$ there exists a $K>0$ such that if $[x_0,x_1]$ is a geodesic connecting $x_0$ to $x_1$ for some $x_0 \in im(\alpha)$, $x_1 \in im(\beta)$ then for each $t \in [0,1]$ if $x_t$ satisfies $d(x_0, x_t)=td(x_0,x_1)$, the function $h$ must satisfy
 $h(x_t) \leq (1-t)h(x_0)+th(x_1)+K.$ 

\end{definition}

A distance-like function is said to be \emph{Morse Convex} if it is $N$-Morse convex for some Morse gauge $N$.

\begin{notation}

We denote the collection of all $N$-Morse convex functions by $H^{N}$ and we let $H= \bigcup\limits_{N \in \mathcal{M}}H^{N}$ where $\mathcal{M}$ is the collection of all Morse gauges.
\end{notation}

\begin{remark}
Since in a $\delta$-hyperbolic space every geodesic is $N$-Morse for a uniform $N$, our definition of $H$ agrees with the definition of horofunctions given in \cite{Coornaert2001} when $X$ is hyperbolic.
\end{remark}

\begin{remark}
We remark that in Definition \ref{Morse-Convex}, the constant $K$ depends only on $h$, $N$, $\alpha$ and $\beta$ but not on $x_0$ or $x_1.$
\end{remark}

\begin{definition}

We define $H^{N}_{p}$ be the subset of $H^{N}$ with an $N$-Morse gradient ray starting at $p$. And we define $(H_{\star})_{p}$ to be the subset of all $h \in H$, such that $h$ has some Morse gradient ray starting at $p$.
\end{definition}

Notice that Lemma \ref{AllMorse} shows that for any $q \in X,$ we have $(H_{\star})_{p}=(H_{\star})_{q}$. Therefore, $(H_{\star})_{p}$ will be simply denoted by $H_{\star};$ the subset of all $h \in H$ such that $h$ has some Morse gradient ray. 

Again, using Lemma \ref{AllMorse} we can see that $H_{\star}= \bigcup_{N \in \mathcal{M}} H^{N}_p$ where $\mathcal{M}$ is the collection of all Morse gauges.

If we let $H^{N}_{p,p}$ be the subspace of all $h \in H^{N}$ with $h(p)=0$, then we have $\overline{H^{N}_{p}} \cong H^{N}_{p,p}$, where $\overline{H^{N}_{p}}$ is the quotient of $H^{N}_{p}$ by the subspace of all constant maps.

The aim of the next few lemmas is to prove that each $N$-Morse convex function $\overline{h}$ defines a unique point at the Morse Boundary, in other words, we will show that if $g_1$ and $g_2$ are two Morse gradient rays, then there must exist some $C \geq 0$ such that $d(g_1(t), g_2(t))<C$ for all $t \in [0, \infty).$ We will also give an example demonstrating that the above statement is not true if we don't restrict our attention to Morse geodesics.

\begin{definition}[thin polygons]
An $n$-polygon is said to be $\delta$-thin if the $\delta$-neighborhood of the of any $n-1$ sides of it contain the other $n$-th side.
\end{definition}

\begin{lemma}[thin Morse-rectangles]\label{thin polygons}
Let $X$ be a proper geodesic metric space and let $[x,z_1]$, and $[y,z_2]$ be $N$-Morse geodeiscs, if $[z_1, z_2]$ is a geodesic connecting $z_1$ to $z_2$, then the polygon $[x,y] \cup [y,z_2] \cup [z_2, z_1] \cup [z_1, x]$ is $\delta$-thin where $\delta$ depends only on $N$ and on $d(x,y).$

\end{lemma}

\begin{proof}
The proof follows easily from Lemma 2.8 in \cite{Cordes2017}.
\end{proof}
The following key lemma states that if $\alpha$ and $\beta$ are two $h$-gradient rays that are Morse, then $d(\alpha(t), \beta(t)) \leq C$ for all $t \in [0, \infty)$. It makes it possible to give a well-defined map from the space of distance-like function to the Morse boundary
\begin{lemma}[Morse-gradient rays fellow travel]\label{unique point}
Let $l \in \overline{H}$, and let $ \alpha, \beta$ be two Morse gradient rays associated to $l$, then $\exists C>0$ such that $d(\alpha(t), \beta(t)) \leq C$ for all $t>0.$
\end{lemma}

Before we give the proof, we give a non-example for the case where $\alpha$ and $\beta$ are not Morse. Let $X$ be the Cayley graph of $\mathbb{Z \oplus Z}$ with the standard generators. Consider the two geodesic rays $\alpha$ and $\beta$ in $X$ given in Figure 4. If we take $h=b_{\alpha}$, the Busemann function of $\alpha$, then $\overline{h}$ is in $\overline{H}$ and both $\alpha$ and $\beta$ are $h$-gradient rays. However, $d(\alpha(t), \beta(t))$ is unbounded.

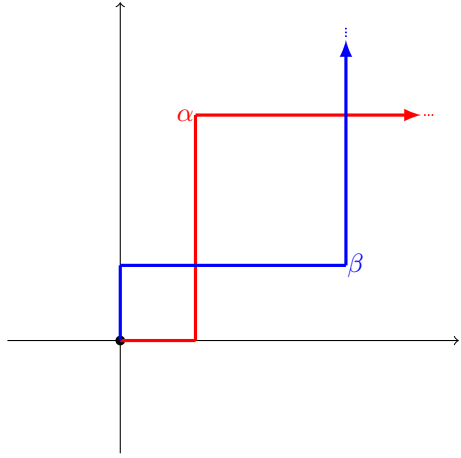
\begin{figure}
    \centering

    \label{fig:nonexample}
    \begin{tikzpicture}[scale=.5]
\draw[thin,->] (-3,0) -- (9,0);
\draw[thin,->] (0,-3) -- (0,9);

\draw[thick,fill=black] (0,0) circle (0.1cm);

\draw[ very thick,red] (0,0) -- ++(2,0);

\draw[ very thick,blue] (0,0) -- ++(0,2);

\draw[ very thick,red] (2,0) -- ++(0,6) node[left] {$\alpha$};

\draw[ very thick,blue] (0,2) -- ++(6,0) node[right] {$\beta$};

\draw[ -latex, very thick,red] (2,6) -- ++(6,0);

\draw [blue] (6,8.1) circle (.01cm); 
\draw [blue] (6,8.2) circle (.01cm);
\draw [blue] (6,8.3) circle (.01cm);

\draw [red] (8.1,6) circle (.01cm); 
\draw [red] (8.2,6) circle (.01cm);
\draw [red] (8.3,6) circle (.01cm);

\draw[-latex,  very thick,blue] (6,2) -- ++(0,6);

\end{tikzpicture}
\caption{The Busemann function of $\alpha$ given by $h=b_{\alpha}$ is in $H,$ and both $\alpha$ and $\beta$ are $h$-gradient rays but $d(\alpha(t), \beta(t))$ is unbounded}

\end{figure}

\begin{proof}
Let $N:=max\{M, M'\}$ where $M$ and $M'$ are the Morse gauges for $\alpha$ and $\beta$ respectively. First we treat the case where $\alpha$ and $\beta$ start at the same point. Let $x:=\alpha(0)=\beta(0)$ and let $\gamma$ be a geodesic connecting $\alpha(t)$ to $\beta(t)$ for some $t \geq 0$. Let $h \in l$ be normalized so that $h(x)=0$. Notice that by assumption, since $h$ is Morse convex, there must exist some $K$ such that $h(\gamma(t))$ is $K$-convex. Let and $m$ is the middle point of $\gamma$, then we have $h(m) \leq \frac{1}{2}h(\alpha(t))+ \frac{1}{2}h(\beta(t))+K=\frac{-t}{2}-\frac{t}{2}+K=-t + K.$ This implies that $d(x,\alpha(t))=t \leq -h(m)+K=h(x) - h(m) \leq d(x,m) +K$. Since $\alpha$ and $\beta$ are both $N$-Morse, then by Lemma 2.2 in \cite{Cordes2017}, the triangle $[x, \alpha(t)] \cup \gamma \cup  [x, \beta(t)]$ is $4N(3,0)$ thin. Therefore, there must exist some point $w \in [x,\alpha(t)] \cup [x, \beta(t)]$ such that $d(m,w) \leq 4N(3,0))$, without loss of generality, assume that $w \in [x, \alpha(t)]$. Notice that since $m$ is the midpoint of $\gamma,$ we must have $d(\alpha(t), \beta(t))=2d(\alpha(t),m) \leq 2(4N(3,0)+d(\alpha(t), w))$. Thus, in order to put a bound on $d(\alpha(t), \beta(t))$, it is enough to put a bound on $d( \alpha(t), w)$. Notice that we have the following
\\
$d(x,w)+d(w, \alpha(t))=d(x, \alpha(t)) \leq d(x,m)+ K \leq d(x,w) +d(w,m) + K \leq d(x,w)+4N(3,0) + K$. Therefore, $d(w, \alpha(t)) \leq 4N(3,0) + K$. This gives that
$d(\alpha(t), \beta(t))=2d(\alpha(t),m) \leq 2(4N(3,0)+d(\alpha(t),w)) \leq 2(4N(3,0)+ 4N(3,0)+K)$

And thus, letting $C=2(4N(3,0)+d(\alpha(t),w)) \leq 2(4N(3,0)+ 4N(3,0)+K$ yields that $d(\alpha(t), \beta(t)) \leq C$ for all $t \in [0, \infty).$

Now we show it more generally, in other words, now we don't assume that $\alpha$ and $\beta$, start at the same point (see Figure 5)). Up to replacing $\alpha$ by some $\alpha'$ defined by $\alpha'(t)=\alpha(a+t)$ for some $a \geq 0$, we may choose a representative $h$ so that $h(\alpha(0))=h(\beta(0))=0$. Let $t>0$ and let $\gamma$ be a geodesic connecting $\alpha(t)$ to $\beta(t)$. Notice that by Lemma \ref{thin polygons} the polygon $[\alpha(0), \beta(0)] \cup [\beta(0), \beta(t)] \cup [\beta(t), \alpha(t)] \cup [\alpha(0), \alpha(t)] $ must be $\delta$-thin where $\delta$ depends only on $d(\alpha(0), \beta(0))$ and on $N$. This implies that the middle point of $\gamma$, denoted by $m$, is $\delta$-close to $[\alpha(0), \beta(0)] \cup [\beta(0), \beta(t)] \cup [\alpha(0), \alpha(t)]$. We claim that if $t$ is large enough, then $m$ can't be $\delta$-close to  $[\alpha(0), \beta(0)]$. First notice that since $h$ is Morse convex, we must have $h(m) \leq \frac{-t}{2}+\frac{-t}{2}+K=-t+K$. This gives that $d(\alpha(0), \alpha(t))=t \leq -h(m)+K=h(\alpha(0))-h(m)+K \leq d(\alpha(0),m)+K$. Suppose for the sake of contradiction that there exist $w \in [\alpha(0), \beta(0)]$ with $d(w,m) \leq \delta$. Let $m'$ be the middle point of $[\alpha(0), \beta(0)]$, without loss of generality, we may assume that $w$ lies between $\alpha(0)$ and $m'$. But this yields: $d(\alpha(0), \beta(0))=2d(\alpha(0), m') \geq 2d(\alpha(0), w) \geq 2(d(\alpha(0), m)-d(m,w)) \geq 2(d(\alpha(0), \alpha(t))-K-\delta) \geq 2t-2K-2\delta$ which is a contradiction since $t$ is arbitrary. This implies the existence of $w \in [\alpha(0), \alpha(t)] \cup [\beta(0), \beta(t)],$ with $d(w,m) \leq \delta$, without loss of generality, assume that $w \in [\alpha(0), \alpha(t)].$ Again, notice that $d( \alpha(t), \beta(t))=2d(\alpha(t), m) \leq d(\alpha(t), w)+d(w,m) \leq d(\alpha(t), w)+ \delta$. Therefore, in order to bound $d(\alpha(t), \beta(t)$, we only need to bound $d(\alpha(t), w)$. But since $d(\alpha(t), w)+d(w, \alpha(0))=d(\alpha(0), \alpha(t)) \leq d(m, \alpha(0)+K \leq d(m,w)+d(w, \alpha(0))+K \leq \delta+d(w, \alpha(0))+K$. This implies that $d(\alpha(t), w) \leq K+ \delta$ which in turns gives us that  $d( \alpha(t), \beta(t))=2d(\alpha(t), m) \leq d(\alpha(t), w)+d(w,m) \leq d(\alpha(t), w)+ \delta \leq K+2\delta$ for all $t$ large enough.

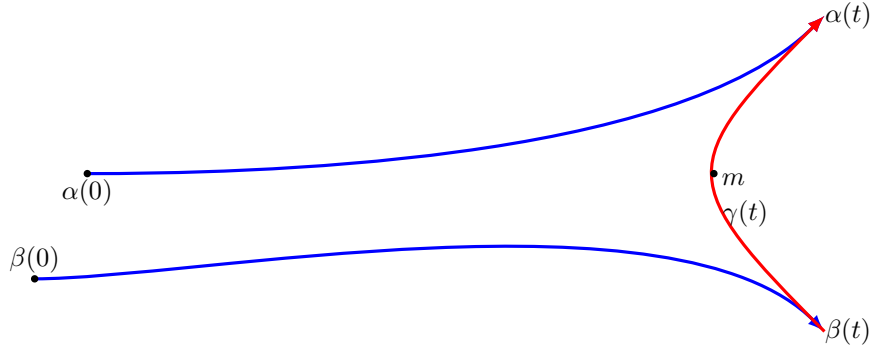
\begin{figure}
    \centering
    
    \label{Main}

\begin{tikzpicture}[scale=.7]

\draw[very thick,blue, latex-] (12,8) .. controls ++({-4*cos(45)},{-4*sin(45)}) and ++(3,0) ..(-2,5);

\draw[very thick,blue, latex-] (12,2) .. controls ++({-4*cos(45)},{4*sin(45)}) and ++(3,0) ..(-3,3);

\draw[very thick,red, latex-] (12,8) .. controls ++({-4*cos(45)},{-4*sin(45)}) and ++({-4*cos(45)},{4*sin(45)})  ..(12,2);
\node[below] at (10.5,4.5) {$\gamma(t)$};

\node[left] at (10.5,4.9) {$m$};
\draw[thick,fill=black] (9.9,5) circle (.05cm);

\draw[thick,fill=black] (-2,5) circle (.05cm);
\node[below] at (-2,4.9) {$\alpha(0)$};

\draw[thick,fill=black] (-3,3) circle (.05cm);
\node[above] at (-3,3.1) {$\beta(0)$};

\node[right] at (12,8) {$\alpha(t)$};

\node[right] at (12,2) {$\beta(t)$};

\end{tikzpicture}

\caption{Morse gradient rays are convergent}

\end{figure}
\end{proof}

We remark that the previous Lemma does not imply that if $h$ is Morse convex then \emph{all} of it's gradient rays converge to the same point. The Lemma only says the Morse gradient rays of $h$ must converge. A natural question to ask at the point would be

\begin{question}
If $h$ is Morse convex and $h$ has some Morse gradient ray. Are all other $h$-gradient rays Morse?
\end{question}

The next few Lemmas aim to show that the space of Morse convex functions $H$ is big enough to surject on the Morse boundary. This is done by showing that for any Morse ray $c$ starting at $p$, we have $b_c \in H^{N}_{p} \subseteq H_{\star}$

\begin{lemma}\label{busemann is convex}
Let $c$ be an $N$-Morse ray starting at $p$, then the Busemann function $b_{c} \in H_{p}^{N}\subseteq H_{\star}.$ 
\end{lemma}

Before we start the proof, we give an easy corollary of Lemma 2.8 in \cite{Cordes2017}.
\begin{corollary}
Let $\alpha$ be some $N$-Morse geodesic starting at $p$ and let $p'$ be any other point in the space $X$. Then, there exist some $N' \geq N$ such that any geodesic starting at $p'$ and ending at $\alpha$ is $N'$-Morse, where $N'$ depends only on $N$ and on $d(p,p').$
\end{corollary}

Now we prove Lemma \ref{busemann is convex}

\begin{proof}
Let $c$ be an $N$-Morse ray starting at some point $p$ and set $h=b_c$. The first thing to show is that $h$ has an $N$-Morse ray starting at $p$, but $c$ clearly does the job. Now, given any two $N$-Morse rays $g_1,$ $g_2$, starting at $p_1$ and $p_2$ respectively, we want to show that there exists some $K>0$ such that if $[x_1,x_2]$ is a geodesic connecting some $x_1 \in im(g_1)$ to $x_1\in im(g_2)$, then for any $s \in [0,1]$, if $x_s$ satisfies $d(x_1,x_s)=sd(x_1,x_2)$, then $h(x_s) \leq (1-s)h(x_1)+sh(x_2)+K.$ Let $t \in [0, \infty)$, we claim that any geodesics $[x_1,c(t)]$, $[x_2,c(t)]$ must be $M$-Morse where $M$ depends only on $N$ and on the distances between the base points $\{p, p_1, p_2\}$. Consider some geodesic triangle with vertices $c(t), p, p_1$. See Figure 6. By the previous corollary, any geodesic $[p_1,c(t)]$ must be $N'$-Morse where $N'$ depends only on $N$ and on $d(p, p_1)$. Now, applying Lemma 2.3 in \cite{Cordes2017}, to a geodesic triangle with vertices $p_1,\,x_1$ and $c(t)$ yields that $[x_1, c(t)]$ must be $N''$-Morse where $N''$ depends only on $N$ and on $d(p,p_1).$ Similarly, we get that a geodesic $[x_2, c(t)]$ is $M'$-Morse, where $M'$ depends only on $N$ and $d(p,p_2)$. Therefore, if $M:=max\{N'',M'\},$ then the geodesics $[x_1, c(t)]$, $[x_2,c(t)]$ are both $M$-Morse. By Lemma 2.2 in \cite{Cordes2017}, we get that the triangle $\Delta$ given by $[x_1, c(t)] \cup [x_2, c(t)] \cup [x_1, x_2]$ is $4M(3,0)$ thin, but then by Lemma \ref{equivalence}, and \ref{convex}, we get that 
$$d(c(t),x_s) \leq (1-s)d(c(t),x_1)+sd(c(t),x_2)+32M(3,0)$$

But then, if we substract $t$ from both sides we get that
$$d(c(t),x_s)-t \leq (1-s)d(c(t),x_1)+sd(c(t),x_2)+32M(3,0)-t$$ or 

$$d(c(t),x_s)-t \leq (1-s)(d(c(t),x_1)-t)+s(d(c(t),x_2)-t)+32M(3,0)$$

Now, if we let $ t\rightarrow \infty$, we get that $b_c(x_s) \leq (1-s)b_c(x_1)+sb_c(x_2)+32M(3,0)$

\begin{figure}
    \centering

\begin{tikzpicture}[scale=.7]

\draw[very thick,blue, latex-] (12,8) .. controls ++({-4*cos(45)},{-4*sin(45)}) and ++(3,0) ..(-2,5);

\draw[very thick,black, -latex] (-4,4)-- ++(17,1);

\node[left] at (-4,4) {$p$};
\draw[thick,fill=black] (-4,4) circle (.05cm);

\draw[very thick,blue, latex-] (12,2) .. controls ++({-4*cos(45)},{4*sin(45)}) and ++(3,0) ..(-3,3);

\draw[very thick,red, latex-] (8,6) .. controls ++({-2*cos(45)},{-2*sin(45)}) and ++({-2*cos(45)},{2*sin(45)})  ..(8,3.5);
\node[above] at (8,6) {$x_2$};
\draw[thick,fill=black] (8,6) circle (.05cm);

\node[above] at (8,3.5) {$x_1$};
\draw[thick,fill=black] (8,3.5) circle (.05cm);

\draw[thick,fill=black] (-2,5) circle (.05cm);
\node[below] at (-2,5) {$p_2$};

\draw[thick,fill=black] (-3,3) circle (.05cm);
\node[above] at (-3,3) {$p_1$};

\node[right] at (12,8) {$g_2(t)$};

\node[right] at (12,5.4) {$c(t)$};

\node[right] at (12,2) {$g_1(t)$};

\end{tikzpicture}

\caption{Busmeann function is Morse convex}
\end{figure}
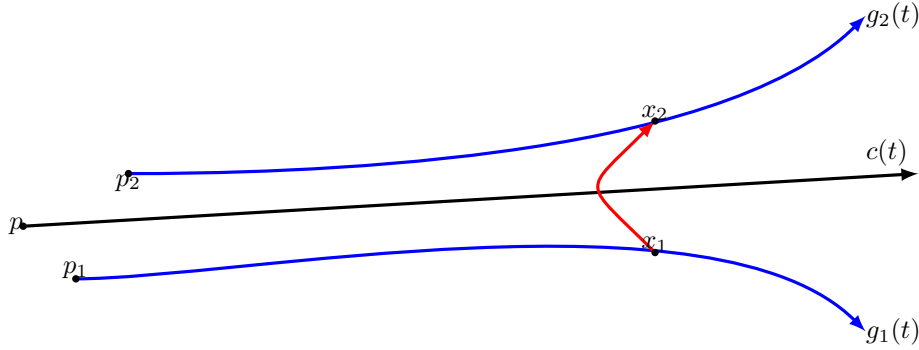

\end{proof}

\begin{comment}
32=4 \times 2 \times 4
\end{comment}

Now, we define the Morse boundary. For more details, see \cite{Cordes2017}

\begin{definition}\label{The Morse Boundary}
let $p \in X$ and let $N$ be a Morse gauge. The $N$-th component of the Morse boundary, denoted by $\partial^{N}X_p$, is defined by

$$\partial^{N}X_p:= \{ [\alpha]\,\,|\,\,\exists \beta \in [\alpha] \text{ that is an } N \text{-Morse ray with } \beta(0)=p \}$$

This set is given the following topology. Fix $p \in X,$ convergence in $\partial^{N}X_p$ is defined as: $x_n \rightarrow x$ as $n \rightarrow \infty$ if there exists $N$-Morse geodesic rays $\alpha_n$ starting at $p$ with $[\alpha_n]=x_n$ such that every subsequene of $\alpha_n$ has a subsequence converging unifromaly on compact sets to some $\alpha$, with $[\alpha]=x.$ Closed sets of $\partial^{N}X_p$ are then defined as: $B \subseteq \partial^{N}X_p$ is closed if and only if whenever $x_n \in B$ with $x_n \rightarrow x,$ we must have $x \in B.$  

$$\partial_{\star}X:=\underrightarrow\lim\partial^{N}X_p$$
where the direct limit is taken over the collection $\mathcal{M}$ of all Morse gauges. We remind the reader that in the direct limit topology a set $U$ is open (or closed) if and only if $U \cap \partial^{N}X_p $ is open (or closed) for all $N \in \mathcal{M}.$
\end{definition}

\begin{theorem}\label{maintheorem}
The natural map $\varphi_{N}:\overline{H^{N}_{p}} \rightarrow \partial^{N}X_p$, is well-defined, continuous, and surjective.  
\end{theorem}

\begin{proof}
Since $\overline{H^{N}_{p}}$ is homeomorphic to $H^{N}_{p,p}$, we will work with the later space instead.  We showed in Lemma \ref{unique point} that if $g_1$ and $g_2$ are two Morse gradient rays associated to the same horofunction $\overline{h} \in \overline{H^{N}_{p}} \subseteq \overline{H_{\star}} $, then $g_1$ and $g_2$ must define the same point at infinity, this shows that $\varphi_{N}$ is well defined. To show that $\varphi_{N}$ is continuous, let $h_n, h \in H^{N}_{p,p} $ with $h_n \rightarrow h$. We need to show that $\varphi_{N}(h_n) \rightarrow \varphi_{N} (h).$ Let $g_n \in \varphi_{N}(h_n)$. In other words, $g_n$ is a sequence of $N$-Morse rays starting at $p$. By definition of convergence in $\partial^{N}X_p$, we need to show that any subsequence of $g_n$, say $g_{n_{k}}$ has some subsequence $g_{n_{k_{i}}}$ that converges uniformly on compact sets to some ray $\beta$, where $\beta \in \varphi_{N}(h)$. Notice that in order to show that $\beta \in \varphi_{N}(h)$ we need to show that $\beta$ is an $h$-gradient ray and that $\beta$ is $N$-Morse. By Arzel`a-Ascoli, any subsequence $g_{n_{k}}$ contains a subsequence $g_{n_{k_{i}}}$ converging to some $\beta$ uniformly on compact sets. Now we are left to show two things, first, we need to show that $\beta$ is a gradient ray for $h$ and then we should show that $\beta$ is in fact $N$-Morse. To show that $\beta$ is an $h$-gradient ray, we need to show that for a given $t_1, t_2 \in [0, \infty)$, we must have $h(\beta(t_1))-h(\beta(t_2))=t_2-t_1$. Let $\epsilon>0,$ we will show that $|h(\beta(t_1))-h(\beta(t_2))-(t_2-t_1)|< \epsilon$ giving the desired conclusion. Notice that the two sets $K_1=\{g_{n_{k_{i}}}(t_1)\}$ and $K_2=\{g_{n_{k_{i}}}(t_2)\}$ live on circles of radius $t_1, t_2$ respectively. Since $X$ is proper, then there exist a compact set $K$ with $K_1, K_2 \subseteq K$. Since $h_n \rightarrow h$ uniformly on compact sets, then $h_n\rightarrow h$ uniformly on $K.$ Therefore, there must exist a positive integer $s_1$ such that for each $i \geq s_1$, we have $|h_{n_{k_{i}}}(g_{n_{k_{i}}}(t_1))-h(g_{n_{k_{i}}}(t_1))|< \frac{\epsilon}{4}$ and $|h_{n_{k_{i}}}(g_{n_{k_{i}}}(t_2))-h(g_{n_{k_{i}}}(t_2))|< \frac{\epsilon}{4}$. Also, since $h$ is uniformly continuous and $g_{n_{k_{i}}}(t_1) \rightarrow \beta (t_1)$ and $g_{n_{k_{i}}}(t_2) \rightarrow \beta (t_2)$, there must exist $s_2$ such that if $i \geq s_2$, then both of the quantities $|h(g_{n_{k_{i}}}(t_1))-h(\beta(t_1))|$ and $|h(g_{n_{k_{i}}}(t_2))-h(\beta(t_2))|$ are less than $\frac{\epsilon}{4}$. Now notice that $|[t_2-t_1]-[h(\beta(t_1))-h(\beta(t_2))]|=|[h_{n_{k_{i}}}(g_{n_{k_{i}}}(t_1))-h_{n_{k_{i}}}(g_{n_{k_{i}}}(t_2))]-[h(\beta(t_1))-h(\beta(t_2))]| \leq |[h_{n_{k_{i}}}(g_{n_{k_{i}}}(t_1))-h(g_{n_{k_{i}}}(t_1))]+[h(g_{n_{k_{i}}}(t_2))-h_{n_{k_{i}}}(g_{n_{k_{i}}}(t_2))]+[h(g_{n_{k_{i}}}(t_1))-h(\beta(t_1))]+[h(\beta(t_2))-h(g_{n_{k_{i}}}(t_2))]| < 4\frac{\epsilon}{4}=\epsilon$. This shows that $\beta$ is a gradient ray for $h$. Now we need to show that $\beta$ is $N$-Morse, but this is Lemma 2.10 in \cite{Cordes2017}.
\end{proof}

We remind the reader that $H_\star$ was defined as $H_{\star}= \bigcup_{N \in \mathcal{M}} H^{N}_p$ where $\mathcal{M}$ is the collection of all Morse gauges. If we give $H_{\star}$ the direct limit topology, we get the following corollary

\begin{corollary}
There exists a continuous surjection $\varphi:\overline{H_{\star}} \twoheadrightarrow \partial_{\star}X.$
\end{corollary}

\begin{remark}
We remark that one can use a very similar argument to the continuity argument given in the above theorem to show that the space $\overline{H^{N}_p}=H_{p,p}^{N}$ is a closed subset of Lip$_{p}(X, \mathbb{R})=\{f \in \text{Lip}(X, \mathbb{R})|\,\,f(p)=0\}$.
\end{remark}

For the case where $X$ is $\delta$-hyperbolic, the analogus space of horofunctions $H$ surjecting on the Gromov boundary is compact, so it is natural to ask whether that is still the case. We in fact prove the following

\begin{corollary} The topological space $\overline{H_{\star}}$ is compact if and only if $X$ is hyperbolic.

\end{corollary}

\begin{proof}
If the space $X$ is $\delta$-hyperbolic, then by the stability Lemma in \cite{BH}, there exists an $N$ such that every geodeisc ray is $N$-Morse, and our definition of $\overline{H_{\star}}$ coincides with the definition of horofunctions given in \cite{Coornaert2001} where they prove such a space is compact ( Proposition 3.9 of \cite{Coornaert2001}). For the converse, if $\overline{H_{\star}}$ is compact, then by continuity of $\varphi$, its image $\partial_{\star}X$ must also be compact, but then by Lemma 3.3 in \cite{Murray2015} and Lemma 4.1 in \cite{CordesDur2017}, the space $X$ must be hyperbolic.
\end{proof}

We conclude this section by remarking that if $G$ is a finitely generated group acting by isometries on a proper geodesic metric space, then $G$ admits a natural action on Lip$(X, \mathbb{R})$, the space of 1-Lipschitz maps from $X$ to $\mathbb{R}$. In fact, we have the following 

\begin{proposition}

If $G$ is a finitely generated group acting by isometries on a proper geodesic metric space $X$, then the set $H_{\star}$ is $G$-invarient.
\end{proposition}

\begin{proof}
We remarked that due to Lemma 4.8, for any $p, q \in X$, we have $(H_{\star})_{p}=(H_{\star})_{q}=H_{\star}$; the subset of all $h \in H$ such that $h$ has some Morse gradient ray. Let $h \in H_{\star}$, his implies that $h$ has some Morse gradient ray $\alpha$. We want to show that $gh$ also have some Morse gradient ray. Remark that since $g$ is an isometry, the geodesic ray given by $g\alpha$ defined by $(g\alpha)(t)=g \alpha(t)$ is also Morse. We also have $(gh)(g\alpha(t))-(gh)(g\alpha(s))=h(g^{-1}g\alpha(t)))-h(g^{-1}g\alpha(s))=h(\alpha(t))-h(\alpha(s))=s-t$
\end{proof}

\begin{corollary} \label{Mostimportant}

The map $\varphi: \overline{H_{\star}} \twoheadrightarrow \partial_{\star}X$ is $G$-equivariant.
\end{corollary}

\begin{proof}
This follows by the previous Proposition and by the definition of the map $\mathbb{\varphi}.$

\end{proof}

\section{A symbolic presentation of the Morse boundary}
In his famous paper \cite{Gromov1987}, Gromov gives a construction describing the boundary of a $\delta$-hyperbolic group $G$ in terms of equivalence classes of ``derivatives" on its Cayley graph. In this section, we develop a similar construction for the Morse boundary. More precisely, Gromov shows that there exists a finite set of symbols, denoted by $\mathcal{A}$, a $G$-equivariant subset $Y \subseteq \mathcal{A}^G$ and a continuous surjection $\varphi:Y  \twoheadrightarrow \partial G$ which is $G$-equivariant. This is called a \emph{symbolic coding} or \emph{symbolic presentation} of $\partial G$. In other words, if a group $G$ acts on a topological space $X$, we say that this action admits a symbolic coding if there exists a finite set $\mathcal{A}$, some $G$-equivariant subset $Y$ of $\mathcal{A}^G$ along with a continuous surjection of $Y$ onto $X$ which is $G$-equivariant.

 The purpose of this section is to use the map developed in the previous section to show that the action of a finitely generated group $G$ on its Morse boundary admits a symbolic coding. Notice that we have already established a $G$-equivariant continuous surjection $\varphi: \overline{H_{\star}}  \twoheadrightarrow     \partial_{\star} G$. Thus, the only part left now is to show that the set $\overline{H_{\star}}$ has a natural identification with some subset $Y_{\star}$ of $\mathcal {A}^{G}$ for some finite set $\mathcal{A}.$
 
\vspace{5mm}

 {\bf Assumption:} \emph{ For the rest of this section, we fix $G=<S>$ to be a finitely generated group and $X=Cay(G,S).$}

\vspace{5mm}

\begin{definition}
Let $\mathcal A$ be a finite set and let $G$ be a finitely generated group. The \emph{shift space}, denoted by $\mathcal{A}^G$, is defined to be the collection of all maps from $G$ to $\mathcal{A}$.
\end{definition}

Notice that the above shift space $\mathcal{A}^G$ admits a natural action of $G$, called the \emph{Bernoulli shift action} given by the following; if $g \in G$ and $ \sigma:G \rightarrow \mathcal{A}$, we have $(g\sigma)(h)=\sigma(g^{-1}h)$.

Given a finitely generated group $G$ and a finite set of symbols $\mathcal{A}$, the shift space $\mathcal{A}^G$ defined above is simply the collection of all decorations of the Cayley graph's vertices by symbols from $\mathcal{A}.$ For example, if $G=\mathbb{Z}^2$ and $\mathcal{A}=\{X,O\}$, then the set $\mathcal{A}^G$ is simply the set of all tilings of the grid $\mathbb{Z}^2$ with $X$ and $O.$ Let $\text{Lip}^{n}(G,H)$ denote the space of all $n$-Lipschitz maps from $G$ to $H$. Now we define the derivative of an $n$-Lipschitz map (this definition was first introduced by Cohen in \cite{Cohen2017}).

\begin{definition}
Let $G,H$ be finitely generated groups and let $f  \in \text{Lip}^{n}(G,H)$. The derivative of $f$, denoted by $df$ is defined to be a map $df:G \rightarrow B(1_H,n)^{S}$ given by
$g \mapsto (s \mapsto f(g)^{-1}f(gs)).$
Denote the collection of all such derivatives by $DL.$ In other words, $DL:=\{df| \,\,f \in \text{Lip}(G,H)\} \subseteq \mathcal{A}^G$ where $\mathcal{A}=B(1_H,n)^{S}.$ 
\end{definition}

\begin{remark}
Notice that each $f$ is completely determined by $df$ and by $f(1)$. In other words, if we define $\text{Lip}_{e}^{n}(G,H)=\{f \in \text{Lip}^{n}(G,H)\,| \,f(e)=0\}, $ then we have a homeomorphism  $$ \text{Lip}_{e}^{n}(G,H) \cong DL$$. 
\end{remark}

\begin{definition}
Let $\mathcal{A}$ be a finite set and let $F$ be a finite subset of $G$. A pattern is a map $\sigma:F \rightarrow \mathcal{A}$. Let $\mathcal{F}:=\{\sigma_i:F \rightarrow \mathcal{A}\}$ be a finite set of patterns all defined on the same finite set $F$, we define the set $X|_{\mathcal{F}}:=\{ \sigma \in \mathcal{A}^G\,|\, (g\sigma)|_{F} \notin \mathcal{F}\,\,\text{for all}\, g \in G\}$. In this case we say that $\mathcal{F}$ is a finite set of \emph{ forbidden patterns}.
\end{definition}

\begin{definition}
A subset $Y$ of $\mathcal{A}^G$ is said to be a \emph{subshift} if it is closed and $G$-equavarient. A subshift $Y$ of $\mathcal{A}^G$ is said to be \emph{a subshift of finite type} if $Y=\text{Clo}(X|_{\mathcal{F}})$ for some $\mathcal{F}$ as above, where $\text{Clo}(X|_{\mathcal{F}})$ denotes the closure of $X|_{\mathcal{F}}$ in $\mathcal{A}^G.$
\end{definition}

The following is an example of a subshift of finite type. If we again take $G=\mathbb{Z}^2$ and take $\mathcal{A}=\{X,O\}$. Define $Y \subset \mathcal{A}^G$ to be the subset of all $\sigma:G \rightarrow \{X,O\}$ such that neither symbol ever shows up three times in a row, horizantally, vertically or diagonally. This subshift corresponds to the set of all possible $\mathbb{Z}^2$ tic-tac-toe games with no winner. We remark that this example appears in David Cohen's research statement.

 In \cite{Cordes2017}, Cohen showed that $DL \subseteq \mathcal{A}^G$ is a subshift of finite type. For the rest of this paper, we specilize to the case where $H=\mathbb{Z}$ and $n=1$. In other words, we consider the space of all $1$-Lipschitz  maps from $G$ to $\mathbb{Z}$; denoted by Lip($G, \mathbb{Z})$ and the derivatives space, $DL=\{df|f \in \text{Lip}(G,\mathbb{Z})\} \subseteq \mathcal{A}^G$ where $\mathcal{A}=\{-1,0,1\}^S$.

%For the rest of this section, $\mathcal{A}$ will be fixed to be $\{-1,0,1\}^S$. Now, the notation for the next few paragraphs might get confusing but the idea is the following. We seek to establish a continuous surjection from some subspace, say $Y^{N}$, of the space of derivatives $DL \subseteq \mathcal{A}^G$ to the $N$-component of the Morse boundary $\partial^{N}X$. We will manage to do so as follows; we have already established a surjection in Theorem \ref{maintheorem} from $\overline{H^{N}_{p}}$ to the $N$-component of the Morse boundary, $\partial^{N}X$. Therefore, all we need to do is to show that there is a natural way of identifying $\overline{H^{N}_{p}}$ with some subspace $Y^N \subseteq DL \subseteq \mathcal{A}^G$. The main idea is that since each $h \in \overline{H^{N}_{p}}$ is in Lip$(X, \mathbb{R})$ (Proposition \ref{Lip}), and since $X=Cay(G,S)$, one can actually show that $\overline{H^{N}_{p}} \subseteq \text{Lip}(G, \mathbb{Z})$; but each element $h \in \text{Lip}(G, \mathbb{Z})$ is completely determined by its derivative $dh \in DL \subseteq \mathcal{A}^G$ and by the value of $h$ at one point. To summarize, we do already have a surjection $\varphi:\overline{H^{N}_{p}}\rightarrow \partial^{N}X$ but since each $h \in \overline{H^{N}_{p}}$ is in Lip$(G, \mathbb{Z})$, it must be completely determined by its dervative $dh \in \mathcal{A}^G$ where $\mathcal{A}=\{-1, 0, 1\}.$ Therefore, $\overline{H^{N}_{p}}$ is naturally identified with a subset $Y^N \subseteq DL \subseteq \mathcal{A}^G$.

\begin{definition}

Let $(H_{\star})_0$ denote the subset of all $h \in H_{\star}$ such that $h$ takes integer values on the vertices of $X.$ In other words, $$(H_{\star})_0 := \{h \in H_{\star}| \,\, \text{and}\, \,h(X^{0}) \subseteq \mathbb{Z}\}$$
\end{definition}

Now let $\overline{(H_{\star})_0}$ denote the quotient of $(H_{\star})_0$ by subspace of constant functions. Notice that since each $h \in (H_{\star})_0 \subseteq \text{Lip}(X, \mathbb{\mathbb{R}}),$ takes integral values on the vertices of $X$, then $(H_{\star})_0$ can be thought of as a subset of $\text{Lip}(G, \mathbb{\mathbb{Z}}).$ Also, the subspace $(H_{\star})_0$ is clearly $G$-equivariant. The next Lemma states that $(H_{\star})_0$ in the above definition is still big enough to surject on the Morse boundary.

\begin{lemma}
For any Morse ray $c$ starting at $e,$ the corresponding Busemann functions $b_c \in (H_{\star})_0.$
\end{lemma}

\begin{proof}
In Lemma \ref{busemann is convex}, we have already shown that $b_c \in H_{\star}.$ Now, since $X=Cay(G,S)$ and the geodesic ray $c$ starts at $e,$ the result follows. 
\end{proof}

Note that in Corollary \ref{Mostimportant}, we established a continuous $G$-equivariant surjection $\varphi: \overline{H_{\star}} \twoheadrightarrow \partial_{\star}X$. Since $\overline{(H_{\star})_{0}} \subseteq  \overline{H_{\star}}$, then by the previous Lemma, if we let  $\varphi_{0}:=\varphi|_{\overline{(H_{\star})_{0}}}$, we get the following

\begin{corollary}
If $\partial_{\star} G$ is the Morse boundary of $G$, then there exists a continuous $G$-equivariant surjection $\varphi_{0}:\overline{(H_{\star})_{0}} \twoheadrightarrow \partial_{\star} G$.
\end{corollary}

Notice that $\overline{(H_{\star})_{0}}$ is homeomporphic to the space $\{h \in (H_{\star})_{0}|\,\,h(e)=0\}\subseteq \text{Lip}_{e}(G, \mathbb{Z})$ and $\text{Lip}_{e}(G, \mathbb{Z}) \cong DL \subseteq \mathcal{A}^{G}$. We get a $G$-equivariant homeomorphism $\psi$ from $\overline{(H_{\star})_{0}}$ to a subset $Y_{\star}\subseteq \mathcal{A}^G.$ Now, if we let $\phi=\varphi_{0}\psi^{-1}$, we get the following

\begin{corollary} The action of $G$ on its Morse boundary admits a symbolic coding. In other words, there exists a finite set $\mathcal{A},$ a $G$-equivariant subset $Y_{\star} \subseteq \mathcal{A}^G,$ and continuous surjection $\phi: Y_{\star} \twoheadrightarrow \partial_{\star}G$ which is $G$-equivariant.
\end{corollary}
It is known that an action of a discrete group on a Polish topologcial space, which is a certain kind of metrizable space, admits a symolic coding (Theorem 1.4 of \cite{Seward2014}). But, the Morse boundary is not metrizable or even second countable \cite{Murray2015}. However, the above corollary states that the action of $G$ on its Morse boundary still admits a symbolic coding.

We end this section with the following questions:

\begin{question}
In the case where $G$ is a hyperbolic group, the analogous map $\phi: Y_{\star} \rightarrow \partial_{\star}G$ is finite to one. Is that still the case here?
\end{question}

\begin{question}
In the case where $G$ is a hyperbolic group, and for some analogous map $\phi: Y_{\star} \rightarrow \partial_{\star}G$, the set $Y_{\star}$ is a subshift of finite type. If $Y_{\star}$ is given the subspace topology, is it a subshift of finite type? I believe the answer is no.
\end{question}

\begin{question}
By the previous theorem, the Morse boundary of a finitely generated group $G$ can be thought of as a subspace of maps from $G$ to a finite set $\mathcal{A}$. Can one give a characterization of Morse rays in terms of those maps in $\mathcal{A}^G?$
\end{question}

In \cite{zalloum2018}, we show that for a finitely generated group $G=<S>$, and for a fixed $D \geq 0$, the langauge of all $D$-contracting geodesics is a regular langauge. In terms of symbolic dynamics, this result can be restated as follows. For a finitely generated group $G=<S>$, and for a fixed $D \geq 0$, the subshift $Y \subseteq S^{\mathbb{Z}}$ consisting of all $D$-contracting geodesic lines is a sofic shift; meaning that $Y=X_{\mathcal{F}}$ where $\mathcal {F}$ is a regular language. A natural question to ask is
\begin{question}
Is $Y$ a subshift of finite type?
\end{question}

\section{The behaviour horofucntions in CAT(0) spaces}
In this section, we assume $X$ is a proper complete CAT(0) space. The goal is to study the behaviour of Busemann functions whose defining rays are Morse.
\begin{definition}[horofunctions] A convex distance-like function is said to be a \emph{horofunction}

\end{definition}

\begin{lemma}\label{one condition}
Let $h$ be a horofunction. Then $h$ must satisfy the following:\\
For any $x_0 \in X$ and any $r>0,$ the function $h$ attains its minimum on the sphere $S_{r}(x_0)$ at a unique point $y$ with $h(y)=h(x_0)-r$
\end{lemma}

\begin{proof}
Let $x_0 \in X$ and let $r>0$. First we show that there exists $y$ on the sphere $S_{r}(x_0)$ with $h(y)=h(x_0)-r,$ but this is Lemma \ref{funny}. Now we show that $h(y)$ is in fact a minumum on the sphere. If there exists some $w \in S_{r}(x_0)$ with $h(w)<h(y)=h(x_0)-r$, then one would have $d(x_0,w)=r<h(x_0)-h(w),$ but that is not possible as $h$ is 1-Lipschitz. Now we are left to show uniqueness. Suppose for the sake of contradiction that $\exists y' \in S_{r}(x_0)$ such that $h(y')=h(y)=h(x_0)-r$ with $y' \neq y.$ Let $[y,y']$ be the unique geodesic segment connecting $y$ and $y'$. Notice that by convexity of the CAT(0) metric, we have $d(x_0,x) \leq r$ for all $x \in [y, y'].$ But since projections to geodesics in a CAT(0) are unique, there must exist some point $x \in [y,y']$ such that $d(x,x_0)<r$. Since $h$ is convex and since $y$ and $y'$ satisfy $h(y)=h(y'),$ we must have that $ h(x) \leq h(y)$. Let $g$ be an $h$-gradient ray with an initial subsegment $[x_0,y]$. Notice that $h \circ g$ is a strictly decreasing function of $t$. Let $t$ be so that $g(0)=x_0$ and $g(t)=y.$ Since $h(x) \leq h(y),$ there must exist some $s \geq t$ with $h(g(s))=h(x).$ Now notice that $r=d(x_0,y)=d(g(0),g(t))=t \leq s=d(g(s), g(0))=s-0=h(g(0))-h(g(s))=h(x_0)-h(x) \leq d(x_0,x)<r$ which is a contradiction.
\end{proof} 

Let $X$ be any metric space, and let $C(X)$ be the collection of all continuous maps $f:X \rightarrow \mathbb{R}.$ Let $C_{\star}(X)$ be the quotient of $C(X)$ by all constant maps. There is a natural embedding $i:X \rightarrow C_{\star}(X)$ by $i(x)=\overline{d(x,-)}$, denote $\hat{X}$ the closure of $i(X)$ in $C_{\star}(X)$. Now let $B$ be the collection of all continuous maps $h:X \rightarrow \mathbb{R}$ satisfying the following three conditions:
\begin{itemize}
    \item $h$ is convex
    \item $h$ is 1-Lipschitz
    \item For any $y_0 \in X$ and any $r>0,$ the function $h$ attains its minimum on the sphere $S_{r}(y_0)$ at a unique point $y$ with $h(y)=h(y_0)-r$
\end{itemize}

\begin{proposition}\label{all equivalent}

Let $X$ be a complete CAT(0) space and fix $x_0 \in X.$ If $h$ is a distance-like function with $h(x_0)=0$ then the following are equivalent:

\begin{enumerate}
  
    \item h is a horofunction
    \item $h \in B$
    
    \item h is a Busemann function for some geodesic ray $c$
    
    \item $h \in \hat{X}-i(X)$

\end{enumerate}
\end{proposition}

\begin{proof}
For $(1) \Longrightarrow (2)$ notice that since $h$ is distance-like, then by Proposition \ref{Lip} it is 1-Lipschitz. Now Lemma $\ref{one condition}$ implies that $h \in B$. The implication $(2) \Longrightarrow (1)$ is by definition of $B$. The equivalence of (2), (3) and (4) is Corollary 8.20 and Proposition 8.22 in the Horofunctions and Busemann functions section of \cite{BH}.
\end{proof}

\begin{remark}\label{constant}
One consequence of the previous proposition is that any horofunction $h$, up to translation by a constant, is of the form $h=b_c$ where $c$ is a geodesic ray. By Corollary 8.20 in the Horofunctions and Busemann functions section of \cite{BH}, up to changing the base point, such a $c$ is unique. Therefore, any horofunction $h$ defines a unique point $\zeta$ at the CAT(0) boundary $\partial X$ and vice versa. This gives that, up to translation by constants, horofunctions and points in the CAT(0) boundary $\partial X$ are in a 1-1 correspondence. Hence, for any point $\zeta \in \partial X$,  we will denote its corresponding horofunction by $h_{\zeta}+C$ or simply by $h_{\zeta}$ if $h$ is normalized to vanish at a fixed base point $x_0$.

\end{remark}

\begin{corollary}\label{important}
 Let $X$ be a proper complete CAT(0) space and let $x_0 \in X$. If $h$ is a horofunction normalized so that $h(x_0)=0$, then there exists some $ p \in X$ such that for any $x \in X,$ $\exists z_x \in [x,p] $ with $$h(x)=d(x,z_x)-d(z_x,p),$$ furthermore, such $z_x$ is unique. 
\end{corollary}

\begin{proof}
By the previous theorem, up to translation by a constant, there must exist a geodesic ray $c$ such that $h=b_c,$ set $p:=c(0).$ Fix some $x \in X$ and notice that since $$h(x)=\lim_{t\to\infty}(d(x,c(t))-d(c(t),p)),$$ the sequence $d(x,c(n))-d(c(n),p)\rightarrow h(x).$ By Remark 4.3, for each $n$, if we consider the triangle with vertices $x,c(n),p$, the internal point $i_{c(n)}$ must satisfy that $d(x,i_{c(n)})-d(i_{c(n)},p)=d(x,c(n))-d(c(n),p)\rightarrow h(x)$. Now notice that the sequence $i_{c(n)} \in B_r(p)$ where $r=d(p,x)$. Since $X$ is proper, $i_{c(n)}$ must have a convergent subsequence $i_{c(n_k)} \rightarrow z$. Since $d(x,-)$ and $d(-,p)$ are continuous, we must have $d(x,i_{c(n_k)}) \rightarrow d(x,z)$ and $d(i_{c(n_k)},p) \rightarrow d(z,p)$. This implies that $d(x,i_{c(n_k)})-d(i_{c(n_k)},p) \rightarrow d(x,z)-d(z,p)$. But since we also have $d(x,i_{c(n_k)})-d(i_{c(n_k)},p) \rightarrow h(x)$, we must have $h(x)=d(x,z)-d(z,p)$, setting $z_x=z$ establishes the stated result. Uniqueness is clear.
\end{proof}
\begin{remark}
Let $X$ be a complete CAT(0) space and let $c$ be a geodesic ray starting at $p$ and fix some $x \in X.$ For each integer $t \in [0, \infty)$, we get a triangle whose vertices are $x,p,c(t)$, and by Remark \ref{tripod}   , we get three internal points $(i_x)_t, (i_p)_t, i_{c(t)}.$ By an argument similar to the one in the previous corollary, we can show that the sequences $(i_x)_t, (i_p)_t, i_{c(t)}$ converge to three points $i_x, i_p, i_\infty$ with the property that $d(x,i_\infty)=d(x, i_p)$, $d(p,i_x)=d(p,i_\infty)$ and $b_{c}(x)=d(x,i_\infty)-d(i_\infty,p)$. Now, Proposition 8.2 of the CAT(0) boundary section in \cite{BH} shows that the sequence of geodesics $[x,c(t)]$ converges to a geodesic ray $c'$ asymptotic to $c$ where $c'$ is the unique such geodesic ray. This motivates the following definition.
\end{remark}

\begin{definition}
Let $X$ be a proper complete CAT(0) space and let $c$ be a geodesic ray in $X$ starting at $p$. Let $x \in X$, and let $c'$ be the unique geodesic ray starting at $x$ and asymptotic to $c$. \emph{The triangle based on $x,p,c$} is defined to be $im(c) \cup im(c') \cup [x,p]$
\end{definition}

The previous remark shows that for any triangle based at infinity, we get three distinguished points, $i_x, i_p, i_\infty$. with $d(x,i_\infty)=d(x, i_p)$, $d(p,i_x)=d(p,i_\infty)$ and $b_c(x)=d(x,i_\infty)-d(i_\infty,p)$. Since $i_{\infty}$ depends on $x$, we will denote it by $(i_ {\infty})_{x}$.

\begin{conjecture}

Let  $c$ be a geodesic ray starting at $p$ in a CAT(0) space $X$. Then corresponding horofunction $h=b_c$ is Morse if and only if there exists a $\delta$ such that for all $x \in X$, we have ($i_ {\infty})_{x}$ in the $\delta$-neighborhood of $im(c)$.
\end{conjecture}

\begin{remark}
The only if direction of the above conjecture is clear since a Morse geodesic in a CAT(0) space is $\delta$-thin, and I believe there should be an easy proof for the if direction as well.
\end{remark}

\begin{definition}
Let $\zeta$ be a point in the CAT(0) boundary $\partial X$ and let $h$ be one of its horofunctions. For each $r \in \mathbb{R},$ we define  the \emph{$r$-horosphere of $h$} by $H_{r}:=h^{-1}(-r)$. A horosphere of $h$ is just an $r$-horosphere of $h$ for some $r$. A horosphere of $\zeta$ is a horosphere of some $\zeta$-horofunction $h$.
\end{definition}

Notice that for any $C \in \mathbb{R}$ both $h$ and $h+C$ have the same set of horospheres.

\begin{definition}

Let $\zeta$ be a point in the CAT(0) boundary $\partial X$ and let $h$ be one of its horofunction. Let $\{H_n\}_{n \in \mathbb{N}}$ be a sequence of horospheres for $h$. We say that the sequence $\{H_n\}_{n \in \mathbb{N}}$ converges to $\zeta$ if for any sequence $\{x_n\}_{n \in \mathbb{N}}$ with $x_n \in H_n$, such that $x_n \rightarrow \eta \in \partial X$, we must have $\eta=\zeta$. 
\end{definition}

\begin{definition}
Let $\zeta$ be a point in the CAT(0) boundary $\partial X$ and let $h$ be one of its horofunction. For each $r \in \mathbb{R},$ we define  the \emph{$r$-horoball of $h$} by $B_{r}:=h^{-1}((-\infty,-r])$. A horoball of $h$ is an $r$-horoball for some $r$. A horoball of $\zeta$ is a horoball of some $\zeta$-horofunction $h$.
\end{definition}

\begin{remark}
Notice that convergence of horospheres is a hyperbolicity phenomena. For example, in $\mathbb{R}^2$ if one takes the geodesic ray $c$ to be the positive $y$-axis, then the associated horospheres, which are all  hyperplanes perpendicular to $c$, do not converge since we have two different sequences (the one in black and the one in green in Figure 7) living on the same horospheres but yet defining different points in $\partial \mathbb{R}^2$. However, as shown in Figure 8, if we consider horoshperes centered at $\zeta$ in the Poincaré disk model, we can see that any convergent sequence living on those horospheres must converge to $\zeta$.
\end{remark}

\begin{figure}
    \centering

    \label{Horospheres in}

\begin{tikzpicture}[scale=.7]
\draw[thin,->] (-5,0) -- (9,0);
\draw[thin, red, ->] (0,0) -- (0,9);

\draw[thin, blue, ->] (-5,1) -- (9,1);
\draw[thin, blue, ->] (-5,2) -- (9,2);
\draw[thin, blue, ->] (-5,3) -- (9,3);
\draw[thin, blue, ->] (-5,4) -- (9,4);
\draw[thin, blue, ->] (-5,5) -- (9,5);
\draw[thin, blue, ->] (-5,6) -- (9,6);
\draw[thin, blue, ->] (-5,7) -- (9,7);

\draw [fill=green] (0,0) circle (.07cm);
\draw [fill=green] (1,1) circle (.07cm);
\draw [fill=green] (2,2) circle (.07cm);
\draw [fill=green] (3,3) circle (.07cm);
\draw [fill=green] (4,4) circle (.07cm);
\draw [fill=green] (5,5) circle (.07cm);

\draw [fill=black] (-2,0) circle (.07cm);
\draw [fill=black] (-2,1) circle (.07cm);
\draw [fill=black] (-2,2) circle (.07cm);
\draw [fill=black] (-2,3) circle (.07cm);
\draw [fill=black] (-2,4) circle (.07cm);
\draw [fill=black] (-2,5) circle (.07cm);

\end{tikzpicture}

\caption{ Horospheres of $\mathbb{R}^2$ do not converge}

\end{figure}
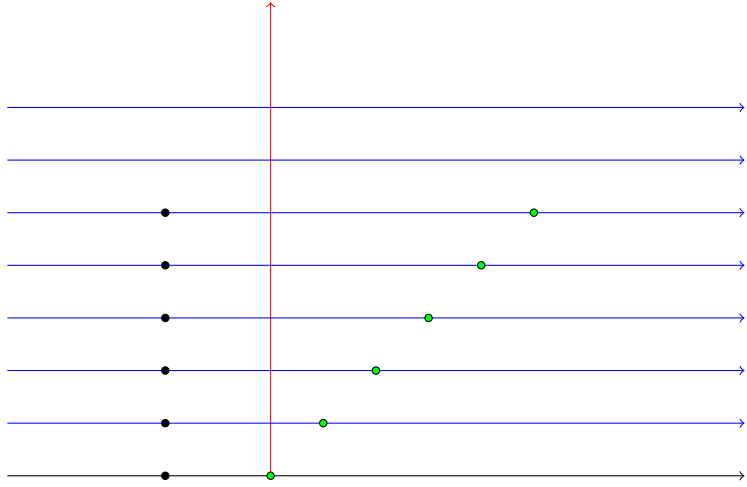

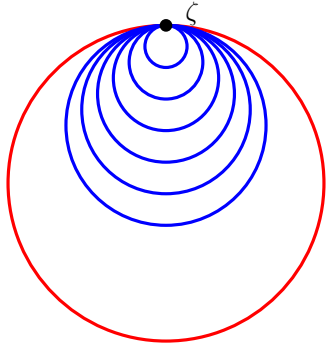
\begin{figure}
    \centering
 
    \label{Horocycl}

\begin{tikzpicture}[scale=.7]

\draw[very thick,red] (0,0) circle (3cm);

%\draw[very thick,blue] (0,3) .. controls ++(0,-2) and ++(-2,0) ..(3,0);%

\draw[very thick,blue] (0,2.6) circle (.4cm);
\draw[very thick,blue] (0,2.3) circle (.7cm);
\draw[very thick,blue] (0,2) circle (1cm);
\draw[very thick,blue] (0,1.7) circle (1.3cm);
\draw[very thick,blue] (0,1.4) circle (1.6cm);
\draw[very thick,blue] (0,1.1) circle (1.9cm);

\draw[thick,fill=black] (0,3) circle (.1cm);

\node[above] at (.5,3) {$\zeta$};

\end{tikzpicture}

\caption{Horospheres in $\mathbb{H}^2$ are convergent}
\end{figure}

\begin{theorem}[Morse horosphers are convergent]
Let $X$ be a CAT(0) space and let $\zeta$ be a point in the Morse boundary of $X$. If $h_{\zeta}$ is a corresponding horofuntion, then any sequence of $h_{\zeta}$-horospheres must converge to $\zeta$. 
\end{theorem}

\begin{proof}
The idea of the proof is the following. First, we will show that if $h_{\zeta}$ is a corresponding horofunction for a Morse point, then, for any horofunction $h$ corresponding to a different point $\eta$ (not necessarily a Morse point) the sum $h_{\zeta}+h{\eta}$ must be bounded below. Once that has been established, the result follows easily, as if a sequence $\{x_n\}_{n \in \mathbb{N}}$ converges to $\eta \neq \zeta$, we would have $h_{\eta}(x_n) \rightarrow -\infty$ and since we also have $h_{\zeta}(x_n) \rightarrow -\infty$, we get that $h_{\zeta}+h_{\eta}$ can not bounded below. Now we give the proof.

Let $\zeta$ be a point in the Morse boundary and let $h_{\zeta}$ be one of its horofunctions. We need to show that any convergent sequence $x_n \in \{H_n\}_{n \in \mathbb{N}}$, where $\{H_{n}\}_{n \in \mathbb{N}}$ are $h_{\zeta}$ horospheres, must converge to $\zeta$. Suppose for the sake of contradiction that $x_n \rightarrow \eta$ with $\eta \neq \zeta$. By \cite{ChSu2014}, the Morse boundary of a CAT(0) space is totally visible, there must exist a geodesic line $c: \mathbb{R}\rightarrow X$ with $c(\infty)=\zeta$ and $c(-\infty)=\eta$. Define $c_1(t)=c|_{[0, \infty)}(t)$ and $c_2(t)=c|_{(-\infty,0]}(-t).$ Fix $p=c(0)$. By Remark \ref{constant}, we have $h_{\zeta}=b_{c_{1}}+C_1$ and $h_{\eta}=b_{c_2}+C_2$. For now, lets assume that $C_1=C_2=0.$ The assumption that $x_n \rightarrow \eta$ implies that for any compact set $K \subseteq X$; and any $\epsilon>0,$ there exists an $n$ such that that $x_n \notin K$ and $x_n \in N_{\epsilon}(im(c_2))$. Since $h_{\eta}$ is distance-like, and up to possibly passing to a subsequence, we get that $h_{\eta}(x_n) \rightarrow -\infty$. But by assumption, $x_n \in H_n$ where $H_n$ are horospheres of $h_\zeta$; this gives that $h_{\zeta}(x_n) \rightarrow -\infty$. Therefore, $h_\zeta+h_\eta$ is not bounded below. Now we argue that this can't be the case. It is an easy exercise to show that for any geodesic line $c:\mathbb{R} \rightarrow X$ and any $x \in X$, we have $b_c(p(x)) \leq b_c(x)$ where $p(x)$ is the projection of $x$ on the closed convex subspace $c(\mathbb{R}).$ Now, notice that for all $t \in \mathbb{R}$, we have $h_{\zeta}(c(t))=-t$ and $h_{\eta}(c(t))=t$. Thus, $h_{\zeta}(c(t))+h_{\eta}(c(t))=0$ for all $t \in \mathbb{R}.$ As noted above, for all $x \in X$, we have $h_{\zeta}(p(x)) \leq h_{\zeta}(x)$ and $h_{\eta}(p(x)) \leq h_{\eta}(x)$. Therefore, we get that $h_{\zeta}(x)+h_{\eta}(x) \geq h_{\zeta}(p(x))+h_{\eta}(p(x))=0$ for all $x \in X$. This show that if $\zeta$ is a Morse point and $\eta$ is any other point in the CAT(0) boundary, then the sum of their corresponding horofunctions, $h_{\zeta}+h_{\eta}$ must be bounded below by zero which is a contradiction. Now if $C_1$ and $C_2$ are not zero, we still get that $h_{\zeta}+h_{\eta}$ is bounded below but now by $-C_1-C_2,$ which is also a contradiction.
\end{proof}

{\bf False-Converse:} \emph{We remark that convergence of horospheres does not characterize horofunctions whose gradient rays are Morse due to the following example.}

\begin{figure}
    \centering
 
    \label{fig:my_label}
\begin{tikzpicture}
\draw[thick] (0,0) ellipse (2cm and 1cm);
\draw[thick] (-.6,0) arc (-180:0:.6cm and .2cm);
\draw[thick] (-.5,-.1) arc (180:0:.5cm and .1cm);
\draw[thick] (2,0) arc (-180:180:.8cm);

\draw[thick,fill=black] (2,0) circle (.04cm);
\node[below] at (0,-1.1) {$T^{2}$};
\node[below] at (2,-1) {$\vee$};
\node[below] at (3,-1) {$S^{1}$};
\end{tikzpicture}

\end{figure}

\begin{example}
Consider the topological space $X$ given by a Torus wedge a circle, $X=T^{2} \vee S^{1}$. The fundemental group of $X$ is given by $G= \langle a,\,b,\,c|\,[a,b] \rangle$, where $a$ and $b$ correspond to the torus's meridian and longitude where $c$ corresponds to the circle. Each edge of the $1$-skeleton of the universal cover inherits a labelling by $a,\, b$ or $c.$ Consider the geodeisc ray $\alpha$ given by $\alpha=aca^2ca^3ca^4c...$ and let $\zeta$ be the point in the CAT(0) boundary defined by $\alpha$. Now, taking $h=b_{\alpha},$ the Busemann function of $\alpha$, we can see that any sequence of $h$-horospheres $\{H_n\}_{n\in \mathbb{N}}$ converges to $\zeta$. However, the point $\zeta$ is not Morse.
\end{example}

\begin{conjecture}
 The CAT(0) assumption in the previous theorem can be dropped.
 
\end{conjecture}

\begin{lemma}\label{unbounded}
Let $\zeta$ be a point in the Morse boundary of a CAT(0) space $X$ and let $h_{\zeta}$ be one of it's horofunctions. If $c$ is a geodesic ray with $c(\infty) \neq \zeta$\, then $h_{\zeta}(c(t)) \rightarrow \infty$ as $t \rightarrow \infty.$
\end{lemma}

\begin{proof}
The proof is easy and left as an exercise for the reader.
\end{proof}

\begin{lemma}
Any horoball is convex.
\end{lemma}

\begin{proof}
This follows easily by convexity of the corresponding horofunction.
\end{proof}

\begin{corollary}
 Let $\zeta$ be a point in the Morse boundary of a proper complete CAT(0) space $X$ and let $B_1$ be a horoball of $\zeta$. If $\eta$ is any other point in the CAT(0) boundary of $X$, and $B_2$ is a horoball of $\eta,$ then $B_{1} \cap B_{2}$ must be bounded.
\end{corollary}

\begin{proof}
We remark that this proof is similar to the one in Proposition 9.35 of \cite{BH}, the key is that the Morse boundary of a CAT(0) space is totally visible. Notice that since $X$ is proper $X \cup \partial X$ must be compact. Since $B_1$ and $B_2$ are horoballs for $\zeta$ and $\eta$ respectively, then by definition, we get two horofunctions $h_1$ and $h_2$ such that $B_1$ is a horoball of $h_1$ and $B_2$ is a horoball of $h_2$. If  $B_1\cap B_2$ is unbounded we get an unbounded sequence $x_n \in B_1 \cap B_2$. But as $X \cup \partial X$ is compact, up to passing to a subsequence, we may assume that $x_n \rightarrow \gamma \in \partial X.$ Notice that by the previous Lemma, $B_1$ and $B_2$ are both convex and therefore $B_1 \cap B_2$ is convex. Hence, if $p \in B_1 \cap B_2$, then $[p,x_n] \subseteq B_1 \cap B_2$ for all $n.$ Thus $[p, x_n]$ converges to a geodesic ray $c \in B_1 \cap  B_2$. Since $c \in B_1,$ $h_1(c(t))$ must be bounded above. Similarly, $h_2(c(t))$ must also be bounded above. But Lemma \ref{unbounded} implies that $c(\infty)=\zeta=\eta$ which is a contradiction.
\end{proof}

\bibliography{ready}{}

\begin{thebibliography}{10}

\bibitem{BH}
Martin~R. Bridson and Andr\'{e} H\"{a}fliger.
\newblock {\em Metric Spaces of Non-Positive Curvature}.
\newblock Springer, 2009.

\bibitem{ChSu2014}
Ruth Charney and Harold Sultan.
\newblock Contracting boundaries of {CAT}(0) spaces.
\newblock {\em Journal of Topology}, 8(1):93--117, sep 2014.

\bibitem{Cohen2017}
David~Bruce Cohen.
\newblock The large scale geometry of strongly aperiodic subshifts of finite
  type.
\newblock {\em Advances in Mathematics}, 308:599--626, feb 2017.

\bibitem{michelcoornaert1993}
Michel Coornaert.
\newblock {\em Symbolic Dynamics and Hyperbolic Groups (Lecture Notes in
  Mathematics)}.
\newblock Springer, apr 1993.

\bibitem{Coornaert2001}
Michel Coornaert and Athanase Papadopoulos.
\newblock Horofunctions and symbolic dynamics on gromov hyperbolic groups.
\newblock {\em Glasgow Mathematical Journal}, 43(03), may 2001.

\bibitem{Cordes2017}
Matthew Cordes.
\newblock Morse boundaries of proper geodesic metric spaces.
\newblock {\em Groups, Geometry, and Dynamics}, 11(4):1281--1306, dec 2017.

\bibitem{CordesDur2017}
Matthew Cordes and Matthew~Gentry Durham.
\newblock Boundary convex cocompactness and stability of subgroups of finitely
  generated groups.
\newblock {\em International Mathematics Research Notices}, aug 2017.

\bibitem{CrokeKleiner}
Christopher~B. Croke and Bruce Kleiner.
\newblock Spaces with nonpositive curvature and their ideal boundaries.
\newblock {\em Topology}, 39(3):549--556, may 2000.

\bibitem{cohen2018}
Yo'av~Rieck David Bruce~Cohen, Chaim Goodman-Strauss.
\newblock Strongly aperiodic subshifts of finite type on hyperbolic groups,
  2017.

\bibitem{zalloum2018}
{Eike} and {Zalloum}.
\newblock Regular languages for contracting geodesics, 2018.

\bibitem{Gromov1987}
M.~Gromov.
\newblock Hyperbolic groups.
\newblock In {\em Essays in Group Theory}, pages 75--263. Springer New York,
  1987.

\bibitem{Maher2015RandomWO}
Joseph Maher and Giulio Tiozzo.
\newblock Random walks on weakly hyperbolic groups.
\newblock 2015.

\bibitem{Murray2015}
Devin Murray.
\newblock Topology and dynamics of the contracting boundary of cocompact cat(0)
  spaces, 2015.

\bibitem{Seward2014}
Brandon Seward.
\newblock Every action of a nonamenable group is the factor of a small action.
\newblock {\em Journal of Modern Dynamics}, 8(2):251--270, nov 2014.

\end{thebibliography}
\bibliographystyle{plain}

\end{document}